\documentclass[reqno]{amsart}
\usepackage{hyperref}

\usepackage{amsmath}
\usepackage{amssymb}
\usepackage{amsthm}

\allowdisplaybreaks
\numberwithin{equation}{section}

\newtheorem{theorem}{Theorem}[section]
\newtheorem{proposition}[theorem]{Proposition}
\newtheorem{lemma}[theorem]{Lemma}
\newtheorem{corollary}[theorem]{Corollary}

\theoremstyle{definition}
\newtheorem{definition}[theorem]{Definition}
\theoremstyle{remark}

\newtheorem*{warning}{Caution}

\newcommand{\I}{{\mathfrak{I}}}
\newcommand{\N}{{\mathbb{N}}}
\newcommand{\R}{{\mathbb{R}}}

\newcommand{\Z}{{\mathbb{Z}}}
\newcommand{\op}{\textit{op}}

\newcommand{\Schw}{\mathcal S}

\newcommand{\LS}{{LS}_\kappa}
\newcommand{\vLS}{{LS}_\varkappa}

\newcommand{\qtq}[1]{\quad\text{#1}\quad}
\newcommand{\eps}{\varepsilon}
\newcommand{\vk}{\varkappa}

\DeclareMathOperator{\supp}{supp}

\DeclareMathOperator{\Tr}{Tr}
\DeclareMathOperator{\Id}{Id}
\DeclareMathOperator{\sech}{sech}

\newcommand{\ddt}{\tfrac{d}{dt}}

\begin{document}

\title{Global well-posedness for the fifth-order KdV equation in \( H^{-1}(\mathbb{R}) \)}
\author{Bjoern Bringmann, Rowan Killip, and Monica Visan}

\address
{Bjoern Bringmann\\
Department of Mathematics\\
University of California, Los Angeles, CA 90095, USA}
\email{bringmann@math.ucla.edu}

\address
{Rowan Killip\\
Department of Mathematics\\
University of California, Los Angeles, CA 90095, USA}
\email{killip@math.ucla.edu}

\address
{Monica Visan\\
Department of Mathematics\\
University of California, Los Angeles, CA 90095, USA}
\email{visan@math.ucla.edu}

\begin{abstract}
We prove global well-posedness of the fifth-order Korteweg-de~Vries equation on the real line for initial data in $H^{-1}(\mathbb{R})$.  By comparison, the optimal regularity for well-posedness on the torus is known to be $L^2(\R/\Z)$.

In order to prove this result, we develop a strategy for integrating the local smoothing effect into the method of commuting flows introduced previously in~\cite{KV} in the context of KdV.  It is this synthesis that allows us to go beyond the known threshold on the torus.
\end{abstract}

\maketitle

\section{Introduction}

In view of its complete integrability, the Korteweg--de Vries equation
\begin{equation}\label{KdV}
\ddt q = - q''' + 6 q q'
\end{equation}
belongs to an infinite hierarchy of commuting flows.  This equation lies second in the hierarchy, after simple spatial translation.  In this paper, we consider the next equation in the hierarchy, namely,
\begin{align}\label{5th}
\ddt q = q^{(5)} -20 q'q'' - 10 qq''' +30 q^2q' .
\end{align}

All flows in the hierarchy describe the evolution of a real-valued field $q$ on the line (or torus), are Hamiltonian, and share the common Poisson structure:
\begin{align}\label{E:PS}
\{ F, G \} = \int  \frac{\delta F}{\delta q}(x) \biggl(\frac{\delta G}{\delta q}\biggr) '(x) \,dx .
\end{align}
In particular, \eqref{5th} is the flow generated by
\begin{align*}
H_\text{5th}(q) := \int \tfrac12 q''(x)^2 + 5 q(x)q'(x)^2 + \tfrac52 q(x)^4\,dx
\end{align*}
and conserves the Casimir $M(q):=\int q(x)\,dx$, as well as
$$
P(q) := \int \tfrac12 q(x)^2\,dx \qtq{and}  H_\text{KdV} (q) := \int \tfrac12 q'(x)^2 + q(x)^3\,dx,
$$
which generate space translations and the KdV flow, respectively.

Due to its place in the KdV hierarchy, the well-posedness problem for \eqref{5th} has received considerable attention.  Until very recently, the best result on the line was global well-posedness in the energy space $H^2(\R)$, which was proved in \cite{MR3096990,MR3301874}.  We also recommend these papers for a discussion of earlier work in this direction, as well as \cite{MR2653659} for a thorough discussion of results in Fourier--Lebesgue spaces.  Unlike for KdV, there is no regularity at which well-posedness can be proved directly by contraction mapping arguments.  This was proved rigorously by Pilod \cite{MR2446185}, who showed that the data-to-solution map is not $C^2$ at  the origin in $H^s(\R)$ for any $s\in\R$.  It was further shown by Kwon \cite{MR2455780} that this map is not uniformly continuous on bounded sets for any $s>0$.  We do not know of any lower bound on well-posedness for \eqref{5th} on the line analogous to those proved for KdV in \cite{MR2830706}.

On the torus, however,  Kappeler and Molnar, \cite{MR3739929}, have obtained an optimal well-posedness result. Concretely, they show that \eqref{5th} is well-posed in $L^2(\R/\Z)$ and conversely, that the data-to-solution map does not admit a continuous extension to $H^s(\R/\Z)$ for any $s<0$.

Very recently, well-posedness of \eqref{5th} in $L^2(\R)$ was proved in \cite{KV} as an application of the method introduced there for the study of KdV.  One virtue of this method is that it applies equally well both on the line and on the torus.  Thus, given the optimality of the Kappeler--Molnar $L^2(\R/\Z)$ result, this appears to be the limit of the method in its original form.  In fact, at that time, it would have been reasonable to guess that the $L^2(\R)$ result might be optimal; indeed, from \cite{MR2267286,KV,MR2830706,MR2927357} we know that the optimal regularity for well-posedness of KdV is the same on the line and on the torus, namely, $H^{-1}$.  On the other hand, for a generic dispersive PDE, we expect the low-regularity behaviour to be better on the line than on the torus due to the improved dispersion --- on the line, high-frequency waves may escape rapidly to spatial infinity, while on the torus, they are trapped.

In this paper, we show that \eqref{5th} is well-posed in $H^{-1}(\R)$; we believe that this is optimal in the scale of $H^s(\R)$ spaces:

\begin{theorem}[Global well-posedness]\label{T:main}
The fifth-order KdV \eqref{5th} is globally well-posed for initial data in \( H^{-1}(\R) \). More precisely, the solution map extends uniquely from Schwartz space to a jointly continuous map \( \Phi \colon \R \times H^{-1}(\R) \rightarrow H^{-1}(\R) \).
\end{theorem}

As noted before, one expects better dispersion for problems posed on the line than for those on the torus.  One expression of this improvement is the local smoothing effect.  First discovered by Kato \cite{MR0759907} in the context of the KdV equation, this is the phenomenon that the solution appears smoother than the initial data if one works locally in space and averages in time.  We will prove the following: 

\begin{theorem}[Local smoothing]\label{T:LS}
For any initial data $q(0)\in H^{-1}(\R)$, the corresponding solution $q(t)$ constructed in Theorem~\ref{T:main} obeys
\begin{equation}\label{E:T:LS}
\sup_{t_0,x_0\in\R} \ \int_{t_0-1}^{t_0+1}\int_{x_0-1}^{x_0+1}  |q'(t,x)|^2 + |q(t,x)|^2 \, dx\,dt  \lesssim \| q(0) \|_{H^{-1}}^2 + \| q(0) \|_{H^{-1}}^{12}.
\end{equation}
Moreover, for every $t_0,x_0\in\R$, the map $q(0)\mapsto q(t-t_0,x-x_0)$ is continuous as a mapping $H^{-1}(\R) \to L^2_t H^1_x ([-1,1]^2)$.  Lastly, $q(t)$ satisfies \eqref{5th} in the sense of spacetime distributions. 
\end{theorem}

Note that this constitutes a gain of two derivatives relative to the initial data, as one would predict from the linear part of the equation.  This result is a close analogue of \cite[Theorem~1.2]{KV}, which gives an analogous gain of one derivative for KdV, and we shall be mimicking the overall structure of that argument in order to prove it.  Nevertheless, the proof of this result involves a considerable jump in complexity relative to the KdV case, resulting from a corresponding jump in the number of subtle cancellations that need to be exhibited.

The distributional nature of the solutions constructed in Theorem~\ref{T:main} follows easily from the earlier parts of Theorem~\ref{T:LS}: That Schwartz solutions are distributional solutions is self-evident; to extend this to all $H^{-1}$ solutions, one need simply rewrite \eqref{5th} as
\begin{equation}\label{5th'}
\ddt q = \partial_x^5 q - 5\partial_x^3 \bigl[q^2\bigr] + \partial_x \bigl[ 5 (q')^2 + 10 q^3\bigr],
\end{equation}
integrate by parts, and employ the continuity shown in Theorem~\ref{T:LS}.

In the KdV setting of \cite{KV}, the local smoothing effect was derived only after the proof of the well-posedness theorem was complete.  This will not work here; we need to use the local smoothing effect in order to go beyond what is possible for the torus.  We refer rather nebulously to the local smoothing effect here, because the crude bound \eqref{E:T:LS} is actually wholly ineffective in helping us prove Theorem~\ref{T:main}; we will need rather more subtle manifestations of this phenomenon.

In order to explain what is to be done in this paper,  it will be helpful if we first endeavor to follow the argument used in \cite{KV}.  A central idea there, which we shall mimic exactly, is the introduction of regularized Hamiltonian flows (depending on a parameter $\kappa$) that have the following properties: (1) They commute with the full flow for all values of $\kappa$. (2) They converge to the full PDE as $\kappa\to\infty$. (3) They are well-posed on $H^{-1}(\R)$.

The construction of the regularized Hamiltonians is inspired by a well-known generating series for the polynomial conserved quantities:
\begin{equation}\label{E:alpha formal}
\alpha(\kappa;q)= \frac{1}{4\kappa^3} P(q) - \frac{1}{16\kappa^5} H_\text{KdV}(q) + \frac{1}{64 \kappa^7} H_\text{5th}(q) + O(\kappa^{-9}).
\end{equation}
The function $\alpha(\kappa,q)$ here is a renormalization of the logarithm of the transmission coefficient (=perturbation determinant) at energy $-\kappa^2$ and is known to be jointly real-analytic on the region where 
\begin{equation}\label{E:Bdelta}
q \in B_\delta := \bigl\{ q \in H^{-1}(\R) : \| q\|_{H^{-1}} <\delta\bigr\}
\end{equation}
and $\kappa\geq 1$, for some fixed $\delta\ll1$; see \cite{KVZ,MR2683250} for details.  Note that the restriction to small $q$ appearing here is actually illusory; \eqref{5th} possesses the scaling symmetry
\begin{equation}\label{E:scaling}
q(t,x)\mapsto q_\lambda(t,x):=\lambda^2 q(\lambda^5 t,\lambda x),
\end{equation}
which means that it suffices to prove Theorems~\ref{T:main} and~\ref{T:LS} for small data.

Rearranging \eqref{E:alpha formal} suggests the definition
\begin{equation}\label{E:H kappa defn}
H_\kappa(q) := 64 \kappa^7 \alpha(\kappa;q) - 16 \kappa^4 P(q) + 4 \kappa^2 H_\text{KdV}(q) .
\end{equation}
Indeed, $H_\kappa(q)\to H_\text{5th}(q)$ as $\kappa\to\infty$ for any Schwartz $q\in B_\delta$. Moreover, for finite $\kappa$ the resulting flow is well-posed on $H^{-1}(\R)$.  To see this, we note that the three constituent Hamiltonians commute; moreover, each is well-posed on $H^{-1}(\R)$.  In the case of $\alpha$, this follows from an ODE argument because $\alpha$ is real analytic on $H^{-1}(\R)$.  Well-posedness of the $P$ flow is trivial since it simply generates spatial translations.  Lastly, well-posedness of the $H_\text{KdV}$ flow on $H^{-1}(\R)$ was the principal result of~\cite{KV}.

Let us now discuss the objective of introducing the $H_\kappa$ flows and explain the role of commutativity.  To do this, it is convenient to adopt a compact notation for the flow of a generic Hamiltonian; concretely, we write
$$
q(t) = e^{tJ\nabla H} q(0) \qtq{for the solution to} \frac{d q}{dt} = \partial_x \frac{\delta H}{\delta q}.
$$
Here $J$ formally represents $\partial_x$, attendant to the Poisson structure \eqref{E:PS}.

In order to prove Theorem~\ref{T:main}, we must show the following: For any $T>0$ and any sequence of Schwartz functions $q_n\in B_\delta$ convergent in $H^{-1}(\R)$, the corresponding sequence of solutions $q_n(t)$ is Cauchy in $C_tH^{-1}([-T,T]\times\R)$.  Evidently,
\begin{align*}
\| q_n(t) -q_m(t) \|_{C_t H^{-1}_x} &\leq \| q_n(t) - e^{tJ\nabla H_\kappa} q_n(0) \|_{C_t H^{-1}_x}  \\
&\qquad +  \| q_m(t) - e^{tJ\nabla H_\kappa} q_m(0) \|_{C_t H^{-1}_x}  \\
&\qquad  + \| e^{tJ\nabla H_\kappa} q_n(0)  - e^{tJ\nabla H_\kappa} q_m(0) \|_{C_t H^{-1}_x}.
\end{align*}
Notice that the last term here converges to zero due to the well-posedness of the $H_\kappa$ flow.  In this way, the proof of Theorem~\ref{T:main} is reduced to showing that the $H_\kappa$ flow closely tracks the full $H_\text{5th}$ flow (at least for $\kappa$ large).

But for the commutativity of the flows, this would not be significant progress.  In view of \cite{MR2455780,MR2446185}, we expect the data to solution map for \eqref{5th} to be very irregular.  Indeed, it is this very irregularity that makes it difficult to conceive of any method by which one may control the difference of two solutions.  By the commutativity of the flows, however, we can write
$$
q_n(t) - e^{tJH_\kappa} q_n(0) = \bigl[e^{tJ(H^\text{5th} - H_\kappa)} - \Id\bigr] \bigl(e^{tJH_\kappa} q_n(0)\bigr) .
$$
Thus, we no longer need to estimate the divergence of two solutions; rather, we merely need to bound the way in which a single solution (generated by the difference of the two Hamiltonians) diverges from its initial data.  The price to pay here is that we must control this difference flow (as we shall call it) for a much richer class of initial data, namely,
\begin{align}
Q := \bigl\{ e^{tJH_\kappa} q_n(0) : n\in\N,\ t\in[-T,T],\ \kappa\geq 1\bigr\}.
\end{align}

We will be able to control the difference flow uniformly on $Q$ because we can show that it is uniformly bounded and equicontinuous in $H^{-1}(\R)$.  These assertions follow readily from the results of \cite{KVZ} in a manner demonstrated already in \cite{KV}.  In fact, the arguments developed in this paper show that $Q$ is precompact in $H^{-1}(\R)$; however, we will not need this in what follows.

We turn now to the heart of the matter, namely, controlling the difference flow.  It is not difficult to write down the equation dictating the evolution of $q$ under this flow; see \eqref{E:H_kappa flow}.  The problem lies in making sense of this equation at such low regularity.  In this regard, the difference flow is no better that the original equation \eqref{5th}; the regularized Hamiltonian $H_\kappa$ only provides good cancellation at low frequencies --- it is regularized!  The first step (appearing already in \cite{KV}) is to make a change of variables, replacing the original unknown $q(t)$ by $2\vk-1/g(\vk;q(t))$.  Here $g$ denotes the diagonal Green's function associated to the Schr\"odinger operator with potential $q$ and $\vk\geq 1$ is an energy parameter, which may be regarded as frozen.

It was shown already in \cite{KV} that (for $q$ small) this change of variables is a diffeomorphism from $H^{-1}(\R)$ to $H^1(\R)$.  The virtue of this change of variables is that it regularizes the nonlinearity.  For example, under the flow \eqref{5th},
\begin{align}\label{E:5th 1/g}
\ddt \, \tfrac{1}{2g(x;\vk,q(t))} &= \partial_x \Bigl(\tfrac{-q''(t,x)+3q(t,x)^2-4\vk^2q(t,x)+8\vk^4}{g(x;\vk,q(t))} \Bigr).
\end{align}

Here we see that the greatest obstruction to making sense of RHS\eqref{E:5th 1/g}, at least as a tempered distribution, is the appearance of $q^2$.  Indeed, this term (which serves as the figurehead of a raft of related problems) is precisely what restricted the analysis in \cite[Appendix]{KV} to treating initial data in $L^2(\R)$.

At first glance the remedy seems obvious (we have announced it already), namely, local smoothing.  Indeed, the \emph{a priori} bound \eqref{E:T:LS} provides more than enough regularity to make sense of $q^2$ for solutions to \eqref{5th}.  But here is the problem: we are trying to control the difference flow, not \eqref{5th}.  Thus, we will need a local smoothing effect for the difference flow.  On the other hand, our ambition is to show that solutions of the difference flow do not move far from their initial data, which seems fundamentally in contradiction to the local smoothing effect --- smoothing happens \emph{because} high-frequencies move away quickly.  Thus, we must complete a delicate balancing act: showing that the difference flow exhibits sufficient high-frequency transport so as to have a local smoothing effect that is strong enough to prove that the difference flow actually transports its initial data a negligible distance in $H^{-1}(\R)$.  Finding a path through this narrow divide is one of the two principal achievements of this paper.  It is accomplished by exhibiting numerous subtle cancelations and by squeezing optimal estimates out of a number of paraproducts that arise throughout the analysis.

The arguments described so far can only show that if $q_\kappa(t)$ is a solution to the difference flow, then
$$
\lim_{\kappa\to\infty} \Bigl\| \phi(x) \Bigl[\tfrac{1}{2g(x;\vk,q_\kappa(t))} - \tfrac{1}{2g(x;\vk,q_\kappa(0))} \Bigr] \Bigr\|_{L^\infty H^{-5} ([-T,T]\times\R)} = 0
$$
for some Schwartz function $\phi$.  The localization $\phi$ is inevitable, because the local smoothing effect is indeed \emph{local}. There is also a considerable loss of regularity compared to the diffeomorphism property, which requires convergence in $H^1$, rather than $H^{-5}$.

The loss of regularity is the lesser problem here and was overcome already in \cite{KV}.  The key observation was that equicontinuous sets of initial data remain equicontinuous under the flow due to known conservation laws (cf. Proposition~\ref{P:H-1 bound}).  This then guarantees that the Green's functions (and their reciprocals) remain equicontinuous, because the mapping of $q$ to $g$ commutes with translations.  One then uses the elementary fact that an $H^1$-equicontinuous sequence converging at lower regularity converges also in $H^{1}$.  Norm convergence of the Green's functions is then transferred to $H^{-1}$ convergence of the $q_n$ via the diffeomorphism property.

In order to complete the proof of Theorem~\ref{T:main}, we need a tightness argument that will upgrade the localized convergence to global convergence.  This is considerably more subtle than equicontinuity and constitutes the second principal achievement of this paper.  Our immediate discussion will focus on how we prove such a property for the solutions $q$ to \eqref{5th}.  Together with boundedness and equicontinuity, this tightness will yield compactness of orbits (over bounded time intervals), which then transfers to the Green's functions via the diffeomophism property.

Given that we are in possession of a microscopic conservation law adapted to regularity $H^{-1}(\R)$, it is natural to imagine that tightness can be proved by simply localizing this conservation law near infinity and controlling the increment, say using local smoothing.  Closer inspection, however, reveals that this would be circular reasoning.  Local smoothing is \emph{proved} by localizing the energy and observing that the dominant term in the increment is coercive.  

The key to breaking this cycle is proving that the high-frequency contribution to the local smoothing norm is small; see Proposition~\ref{P:LS5}.  The low-frequency contribution is controlled by using the global $L^\infty_t H^{-1}_x$ bound.

Proposition~\ref{P:LS5} also plays a key role in proving continuous dependence on the initial data in the local smoothing norm; see Theorem~\ref{T:LS}.  Specifically, it reduces our attention to the low-frequency contribution, whose continuity follows from that shown in Theorem~\ref{T:main}.

Taken together, our boundedness, equicontinuity, and tightness arguments show a strong compactness phenomenon for \eqref{5th}: Given a bounded time interval $[-T,T]$ and an $H^{-1}(\R)$-precompact set of initial data, the corresponding orbits all lie inside a compact subset of $H^{-1}(\R)$.  This is weaker than saying that the orbits are precompact in $C_t H^{-1}([-T,T]\times\R)$, which will follow from Theorem~\ref{T:main}.  The missing ingredient is equicontinuity in time (cf. the Arzela--Ascoli Theorem).  This is one of two key roles played by the difference flow.  The second is guaranteeing the uniqueness of subsequential limits (which exist by compactness) of sequences of solutions with convergent initial data.

An alternate approach to showing the uniqueness of subsequential limits would be to identify a suitable collection of properties that intrinsically and uniquely identify the solutions constructed in this paper. This is an important open problem.  Note that we regard the solutions we construct as canonical, since by Theorem~\ref{T:main} no distinct notion of solution can lead to well-posedness.  For further musings on this question (framed in the KdV setting), see \cite{WN}.

The paper is organized as follows:  We begin Section~\ref{S:2} by introducing notation and reviewing some basic estimates.  We then move on to a review of the material developed in \cite{KV} that we will need; see Subsection~\ref{SS:diagonal}.  We then develop a variety of commutator estimates, which in turn inspire our choice of a family of local smoothing norms; see Definition~\ref{D:LS}.

Section~\ref{S:para} is devoted to the analysis of various paraproducts that arise in the subsequent analysis.  An guiding principle here is that we must obtain large negative powers of the frequency parameter $\vk$ in all our estimates. This is only possible (at the low regularity at which we work) by the proper use of the local smoothing norm.  

The centerpiece of Section~\ref{S:LS} is the proof of Proposition~\ref{P:LS5}.  As discussed above, this goes beyond merely providing the basic local smoothing estimate \eqref{E:T:LS}.  It also demonstrates that the high-frequency contribution to the local smoothing norm is small (for equicontinuous sets of initial data).  This is then deployed in Section~\ref{S:C} to prove compactness of trajectories; see Proposition~\ref{P:compact}.

In Section~\ref{S:dLS}, we prove local smoothing for the difference flow.  This is Proposition~\ref{P:dLS}, which is then used in Corollary~\ref{C:dLS} to control the divergence of the $H_\kappa$ and $H_\text{5th}$ flows.  The analysis in this section is relatively short, because we rely on numerous cancellations exhibited earlier in the analysis, particularly in Section~\ref{S:LS}.

The paper ends with Section~\ref{S:WP} which brings together all the foregoing analysis to complete the proofs of Theorems~\ref{T:main} and~\ref{T:LS}.

\subsection*{Acknowledgements} R.~K. was supported by NSF grants DMS-1600942 and DMS-1856755.  M.~V. was supported by NSF grant DMS-1763074.

\section{Preliminaries}\label{S:2}

We begin by reviewing our basic notation and a few elementary results.

Unless indicated otherwise, spacetime norms are taken over the slab $[-1,1]\times\R$:
$$
\| q \|_{L^p_t H^s} = \bigl\| \| q(t) \|_{H^s(\R)} \bigr\|_{L^p(dt;[-1,1])}
$$
In addition to the usual $H^s$ spaces, we employ the notation
\begin{equation}\label{E:Hs}
\| f \|_{H_\vk^s (\R)}^2 := \int (\xi^2+4\vk^2)^{s} |\hat f(\xi)|^2 \,d\xi,
\end{equation}
where our convention for the Fourier transform is
\begin{equation*}
\hat f(\xi) = \tfrac{1}{\sqrt{2\pi}} \int e^{-i\xi x} f(x)\,dx \qtq{so that} \| \hat f \|_{L^2} = \| f \|_{L^2} . 
\end{equation*}
The particular formulation of \eqref{E:Hs}, including the factor of $4$, is explained by its appearance in \eqref{E:L:HS}. 

We will frequently use the elementary facts
\begin{align}\label{elem mult}
\| w f \|_{H^{\pm 1}_\vk} \lesssim \bigl\{ \| w \|_{L^\infty} + \| w' \|_{L^\infty} \bigr\} \| f \|_{H^{\pm1}_\vk}
\qtq{and} \| w f \|_{H^{\pm 1}_\vk} \lesssim \| w \|_{H^{1}}^{ } \| f \|_{H^{\pm1}_\vk} .
\end{align}
Note that by duality, the results for $H^{-1}_\vk$ are equivalent to those for $H^1_\vk$.

We will use the Littlewood--Paley decomposition extensively in our treatment of paraproducts.  It is based on a smooth partition of unity in Fourier space:  Fix $\varphi:\R\to[0,1]$ that is $C^\infty$ and satisfies $\varphi(\xi)=1$ if $|\xi|\leq 1$ and $\varphi(\xi)=0$ if $|\xi|\geq 2$.  We then define Littlewood--Paley pieces
as follows:
$$
\widehat{f_N}(\xi) = \begin{cases} \varphi(\xi) \hat f(\xi) & \text{: if $N=1$,}\\
		[\varphi(\xi/N)-\varphi(2\xi/N)] \hat f(\xi) &\text{: if $N\in\{2,4,8,\ldots\}$}. \end{cases}
$$
Evidently, $f=\sum f_N$ where the sum is over $N\in\{1,2,4,8,\ldots\}$.

We write $\I_p$ for the Schatten classes (= trace ideals) defined over the Hilbert space $L^2(\R)$, with the particular convention that $\I_\infty$ denotes the space of bounded operators endowed with the operator norm.  This differs from the text \cite{MR2154153} where $\I_\infty$ denotes compact operators (with the same norm).  These classes of operators obey the H\"older inequality in the form
$$
\| A B C \|_{\I_p} \leq \| A \|_{\I_{p_1}} \| B \|_{\I_{p_2}} \| C \|_{\I_{p_3}}\qtq{whenever} \tfrac1p=\tfrac1{p_1} + \tfrac1{p_2} + \tfrac1{p_3}.
$$

Throughout the paper, $R_0(\vk)$ denotes the resolvent
\begin{equation}\label{E:R0 kernel}
R_0(\vk) = (-\partial_x^2 +\vk^2)^{-1} \qtq{with kernel} \langle \delta_x, R_0(\vk) \delta_y\rangle = \tfrac{1}{2\vk} e^{-\vk|x-y|} .
\end{equation}
We shall only consider $\vk\geq 1$ and so $R_0(\vk)$ is positive definite. By its square-root, we shall always mean the positive definite square-root.

The natural compactness criterion for subsets of $H^{-1}(\R)$ is easily intuited from the classical case of $L^p(\R^d)$ settled already by Kolmogorov, Tamarkin, and Riesz (cf.~\cite{Riesz}).  As in this classical case, the following is readily proved by using smooth mollification and smooth truncation to reduce matters to the Arzela--Ascoli Theorem:

\begin{lemma}\label{L:Riesz}
A bounded subset $Q\subseteq H^{-1}(\R)$ is precompact in $H^{-1}(\R)$ if and only if it is both \emph{equicontinuous}, which is to say
\begin{align}\label{E:equi defn}
\lim_{N\to\infty}\ \sup_{q\in Q}\ \int_{|\xi|\geq N} \frac{|\hat q(\xi)|^2\,d\xi}{\xi^2+4} =0,
\end{align}
and \emph{tight}, which means that
\begin{align}\label{E:tight defn}
\lim_{R\to\infty}\ \sup_{q\in Q}\ \sup\{ \langle f, q\rangle : \|f\|_{H^1}\leq 1 \text{ and } \supp(f)\subseteq \R\setminus[-R,R]\bigr\} =0.
\end{align}
\end{lemma}

Evidently, the equicontinuity and tightness criteria could also be formulated using a smooth cutoff (to large values of $\xi$ and $x$, respectively). Although a sharp Fourier cutoff is acceptable, one cannot use a sharp spatial cutoff because this is not a bounded operator in $H^{-1}(\R)$.

\subsection{The diagonal Green's function}\label{SS:diagonal}
This subsection is primarily devoted to recounting material from \cite{KV}, which can be consulted for further details.  A workhorse of the analysis therein is the following computation:

\begin{lemma}\label{L:HS}
For $q \in H^{-1}(\R)$,
\begin{equation}\label{E:L:HS}
\bigl\| \sqrt{R_0(\vk)} q  \sqrt{R_0(\vk)}\bigr\|_{\mathfrak{I}_2}^2 = \tfrac{1}{\vk} \int \tfrac{|\hat{q}(\xi)|^2}{\xi^2+4\vk^2}\, d\xi
= \tfrac1\vk \| q \|_{H^{-1}_\vk}^2.
\end{equation}
\end{lemma}

For $q\in B_\delta$, $\vk\geq 1$, and $\delta>0$ sufficiently small, this lemma guarantees that it is possible to use a Neumann series to construct the resolvent
\begin{equation}\label{E:Neumann}
\!\!R(\vk)=\bigl( -\partial^2 + q +\vk^2\bigr)^{-1} = \sum_{\ell=0}^\infty (-1)^\ell \sqrt{R_0(\vk)} \bigl( \sqrt{R_0(\vk)} q \sqrt{R_0(\vk)}\, \bigr)^\ell\! \sqrt{R_0(\vk)}.
\end{equation}

Using this series, one may show that $R(\vk)$ admits a continuous integral kernel.  Restricting this to the diagonal, yields what we term the \emph{diagonal Green's function}:
\begin{equation}\label{E:g defn}
g(x;\vk,q) := \langle \delta_x, R(\vk) \delta_x\rangle = \tfrac{1}{2\vk} + \sum_{\ell=1}^\infty h_\ell(x;\vk,q)
\end{equation}
where
\begin{equation}\label{E:h defn}
h_\ell(x;\vk,q) := (-1)^\ell \bigl\langle \sqrt{R_0(\vk)} \delta_x,\  \bigl( \sqrt{R_0(\vk)} q \sqrt{R_0(\vk)}\, \bigr)^\ell
	\sqrt{R_0(\vk)}\delta_x\bigr\rangle.
\end{equation}

Arguing by duality and using Lemma~\ref{L:HS}, one readily sees that
\begin{align}\label{E:h in H1}
\| h_\ell \|_{H^{1}_\vk} & \leq \vk^{-(\ell+1)/2} \| q \|_{H^{-1}_\vk} ^\ell \quad\text{for all $\ell\geq 1$.}
\end{align}
From \eqref{E:h in H1} and Cauchy--Schwartz in Fourier space, we deduce that
\begin{equation}\label{E:h in Linfty}
\| h_\ell \|_{L^\infty}  \leq \vk^{-(\ell+2)/2} \| q \|_{H^{-1}_\vk} ^\ell ,
\end{equation}
which shows that $g$ will be non-vanishing for all $\vk\geq 1$ if $\delta>0$ is sufficiently small.

From these estimates and the  inverse function theorem, one can then show:

\begin{lemma}[Diffeomorphism property]\label{L:diffeo} For $\delta>0$ sufficiently small, both
$$
q\mapsto g-\tfrac1{2\vk}\qtq{and} q\mapsto 2\vk - \tfrac1{g}
$$
are diffeomorphisms from $B_\delta$ into $H^1_\vk(\R)$ for every $\vk\geq 1$; both map zero to zero.
\end{lemma}

A natural prerequisite for well-posedness in $H^{-1}(\R)$ is an \emph{a priori} bound on this norm, say, for Schwartz solutions.  In the case of KdV, such results were proved in \cite{KVZ,KT,MR2683250}; these can be readily adapted to \eqref{5th} due to their presence in the same integrable hierarchy.  In order to prove local smoothing, however, it will be essential for us to have a microscopic conservation law attendant to this low regularity.  In this paper, we will use the conserved density
\begin{equation}\label{E:rho defn}
\rho(x;\vk,q) := 2\vk^2 - \tfrac{\vk}{g(x;\vk,q)} + 4 \vk^2 [R_0(2\vk) q] (x),
\end{equation}
which is shown in \cite{KV} to be positive and integrable for $q\in B_\delta$, $\vk\geq 1$, and $\delta>0$ sufficiently small.  Moreover, for such parameters we have
\begin{equation}\label{E:rho alpha}
\int_\R \rho(x;\vk,q) \,dx  = 2\vk\alpha(\vk,q) \approx \int_\R \frac{|\widehat{q}(\xi)|^2}{\xi^2+4\vk^2} \, d\xi.
\end{equation}

\begin{warning}
The definition of $\rho$ here differs from that used in \cite{KV} by the numerical factor $2\vk$.  This change connects $\rho$ more closely with the $H^{-1}_\vk$ norm and will simplify our analysis of local smoothing estimates in a comparable way.
\end{warning}

The quantity $\alpha(q;\vk)$ appearing in \eqref{E:rho alpha} coincides with the renormalized logarithm of the perturbation determinant discussed in \cite{KVZ,MR2683250}; its asymptotic expansion appeared already in \eqref{E:alpha formal}.  In this way, the conservation of $\alpha$ under \eqref{5th} can be inferred from the general framework of the inverse scattering technique.  Alternately, one may deduce this from the associated microscopic conservation law derived in \cite[Appendix]{KV}.  Concretely, for Schwartz solutions to \eqref{5th},
\begin{equation}\label{mscl}
\partial_t \rho + \partial_x j_\text{5th} = 0,
\end{equation}
where
\begin{align}\label{E:rho dot 5}
j_\text{5th} &= - \tfrac{2\vk}{g(\vk)} [ 16\vk^5 g(\vk) - 8 \vk^4 + 4 \vk^2 q + q'' - 3 q^2 ] \notag \\
&\quad -  4\varkappa^2 R_0(2\varkappa) \big[q^{(4)} - 5 (q^2)'' + 5 (q')^2 + 10 q^3\big];
\end{align}
indeed the first term here originates in the time derivative of the reciprocal of the Green's function, while the second term arises by substituting \eqref{5th} into the last term in \eqref{E:rho defn}.

It will also be important for us to have analogous information regarding the $H_\kappa$ flow, that is, the flow induced by the Hamiltonian \eqref{E:H kappa defn} via the symplectic structure \eqref{E:PS}.  Under this flow,
\begin{equation}\label{E:H_kappa flow}
\ddt q = \bigl\{ - 64 \kappa^7 g(\kappa) + 32\kappa^6 - 16 \kappa^4 q + 4 \kappa^2 \bigl[-q'' + 3 q^2\bigr]\bigr\}'
\end{equation}
and correspondingly, $\partial_t \rho + \partial_x j_\kappa = 0$ with
\begin{align}\label{E:rho dot k}
j_\kappa &=  \tfrac{32\kappa^7\vk}{\kappa^2-\varkappa^2} \tfrac{g(\kappa)}{g(\varkappa)} 	- 8 \kappa^2\vk \tfrac{2\kappa^2+2\vk^2 - q}{g(\vk)}
	- \tfrac{32\kappa^2\vk^6}{\kappa^2-\varkappa^2} \notag \\
&\quad - 16 \kappa^2 \varkappa^2 R_0(2\varkappa) \bigl[ -16\kappa^5 g(\kappa) +8\kappa^4 - 4 \kappa^2 q - q'' + 3 q^2 \bigr].
\end{align}
Once again, we have grouped the terms according to their origin.  It is important here to distinguish between $\vk\geq1$, which denotes the energy parameter in $\rho$, and $\kappa\geq 1$ which describes the flow under consideration.

As remarked earlier, these computations provide an elementary justification for the conservation of $\alpha(q)$ under these flows and so, by \eqref{E:rho alpha}, of the following:

\begin{proposition}[A priori bound]\label{P:H-1 bound}
For $\delta>0$ sufficiently small, 
\begin{equation}\label{E:P:H-1}
\bigl\| e^{tJ\nabla H_\kappa+ s J\nabla H_\text{5th}+\tau J\nabla P} q \bigr\|_{H^{-1}_\vk(\R)} \approx \| q \|_{H^{-1}_\vk(\R)},
\end{equation}
uniformly for $t,s,\tau\in\R$, $\kappa,\vk\geq 1$, and $q\in B_\delta\cap\Schw(\R)$.  Moreover, if $Q\subset B_\delta\cap\Schw(\R)$ is $H^{-1}$-equicontinuous, then so is
$$
Q_* = \bigl\{ e^{tJ\nabla H_\kappa+ s J\nabla H_\text{5th}+\tau J\nabla P} q : t,s,\tau\in\R,\ \kappa\geq 1,\text{ and }q\in Q\bigr\}.
$$
\end{proposition}

The equicontinuity claim here follows directly from \eqref{E:P:H-1} due to the uniformity in $\vk$; this is discussed in Section~4 of \cite{KV}.

It will be essential for our analysis to understand the large-$\vk$ behavior of $g(x;\vk)$.  Our next lemma provides important information in this direction.  The exact formulation of the identities that follow is dictated by the need to exhibit certain cancellations later on.  At this moment, however, it is instructive to imagine that our goal is to show that the RHS\eqref{E:H_kappa flow} converges to the RHS\eqref{5th'} as $\kappa\to\infty$.  We do not include a corresponding expression for the cubic terms because it is so complicated as to be of little use in the subsequent analysis.  Although the key identities \eqref{h1 nice} and \eqref{h2 nice} were derived already in \cite{KV}, this result does not appear there and so we include a proof.

\begin{lemma}\label{L:series expansion}
For $q \in H^{-1}(\R)$ and $\vk\geq 1$, we have
\begin{align}
4\vk^2\bigl[ 16 \vk^5 h_1 + 4\vk^2 q + q'' \bigr] &= 4\vk^3 h_1^{(4)} = -q^{(4)} +  \vk h_1^{(6)} \label{E:h1 exp} \\
16 \vk^5 h_2 + 3 \vk^2 (h_1'')^2 - 3q^2 &= - 4 \vk^4 \bigl[ 5 (h_1')^2 - 5\partial_x^2 (h_1^2) \bigr] \notag \\
&\quad  + 4 \vk^4 \partial_x^2 R_0(2\vk) \bigl[(h_1')^2 + 2\partial_x^2 (h_1^2)\bigr].  \label{E:h2 exp}
\end{align}
\end{lemma}

\begin{proof}
From the explicit kernel \eqref{E:R0 kernel} for $R_0$, we deduce 
\begin{equation}\label{h1 nice}
h_1(x) = - \langle \delta_x , R_0(\vk) q R_0(\vk) \delta_x \rangle =- \tfrac{1}{\vk} R_0(2\vk) q (x) .
\end{equation}
Thus, the claims \eqref{E:h1 exp} follow from the symbol identities
\begin{equation*}
4\vk^2\bigl[ - \tfrac{16 \vk^4}{\xi^2+4\vk^2} + 4\vk^2 -\xi^2 \bigr]  = - \tfrac{4\vk^2\xi^4}{\xi^2+ 4 \vk^2} = -\xi^4 + \tfrac{\xi^6}{\xi^2+4\kappa^2}.
\end{equation*}

We turn now to the quadratic term $h_2$, for which we have (cf. \cite[Appendix]{KV}),
\begin{align}\label{h2 nice}
\widehat{h_2}(\xi) 
&= \frac{1}{2\vk \sqrt{2\pi}} \int_\R \frac{\xi^2+ (\xi-\eta)^2+\eta^2+24\vk^2}{(\xi^2+4\vk^2)((\xi-\eta)^2+4\vk^2)(\eta^2+4\vk^2)} \hat{q}(\xi-\eta) \hat{q}(\eta) \, d\eta.
\end{align}
We note that the definition \eqref{E:h defn} yields a double-integral representation for $\hat h_2$; to obtain \eqref{h2 nice}, one then needs to integrate out one variable.  This is easily done by the method of residues, for example.

In this way, \eqref{E:h2 exp} reduces to the algebraic identity
\begin{align*}
&\quad\frac{8\vk^4 [\xi^2+ (\xi-\eta)^2+\eta^2+24\vk^2]}{(\xi^2+4\vk^2)((\xi-\eta)^2+4\vk^2)(\eta^2+4\vk^2)} + \frac{3 \eta^2(\xi-\eta)^2}{((\xi-\eta)^2+4\vk^2)(\eta^2+4\vk^2)}  - 3 \\
&= - \frac{20\vk^2[-\eta(\xi-\eta)+\xi^2]}{((\xi-\eta)^2+4\vk^2)(\eta^2+4\vk^2)} + \frac{4\vk^2\xi^2[\eta(\xi-\eta)+2\xi^2]}{(\xi^2+4\vk^2)((\xi-\eta)^2+4\vk^2)(\eta^2+4\vk^2)}. \qedhere
\end{align*}
\end{proof}

\subsection{Commutator estimates}\label{SS:commutator}

In this subsection, we present several commutator estimates, which will then inform our definition of a local smoothing norm. 

\begin{lemma}[Basic commutator identity]\label{L:basic_commutator}
\begin{align*}
[R_0, \phi ] &= R_0 (\phi'\partial + \partial \phi') R_0 = R_0 (2\phi'\partial + \phi'') R_0 = R_0 (2\partial\phi' - \phi'') R_0.
\end{align*}
\end{lemma}

\begin{proof}
Follows directly from $[-\partial^2+\vk^2,\phi] = - \partial\phi' - \phi'\partial$.
\end{proof}

\begin{lemma}[Multiplicative commutation]\label{L:w commutator}
Assume $w:\R\to (0,\infty)$ satisfies
\begin{equation}\label{E:w hyp}
|w''(x)| + |w'(x)|\lesssim w(x) \qtq{and} \frac{w(y)}{w(x)}\lesssim e^{|x-y|/2}
\end{equation}
uniformly for $x,y\in\R$.  Then
\begin{equation}\label{w under over}
\bigl\| w \partial R_0(\vk) \tfrac{1}{w} \bigr\|_{L^p\to L^p} \lesssim \vk^{-1} \qtq{and} \bigl\| w R_0(\vk) \tfrac{1}{w} \bigr\|_{L^p\to L^p} \lesssim \vk^{-2}
\end{equation}
for every $1\leq p\leq\infty$.  Moreover, the operator $B=B(\vk)$ defined by the identity
\begin{equation}\label{E:w commutator}
w( x) R_0(\vk) = \sqrt{R_0(\vk)} (\Id+B) \sqrt{R_0(\vk)}\, w( x) \qtq{satisfies} \| B\|_\op \lesssim \vk^{-1}
\end{equation}
uniformly for $\vk\geq1$.
\end{lemma}

\begin{proof}
The estimate \eqref{w under over} follows from Schur's test by using the explicit kernel for $R_0$ and just the second inequality in \eqref{E:w hyp}.

Let us now consider the operators
\begin{equation}
B(z) := (-\partial^2+\vk^2)^z R_0(\vk) (w'\partial + \partial w') R_0(\vk) \tfrac{1}{w} (-\partial^2+\vk^2)^{1-z} .
\end{equation}
Note that by Lemma~\ref{L:basic_commutator}, this reduces to $B$ when $z=1/2$.  Our goal is to bound $B(1/2)$ via complex interpolation.  As imaginary powers of positive operators are unitary, this reduces the proof of \eqref{E:w commutator} to showing that
$$
\| B(0) \|_\op + \| B(1) \|_\op   \lesssim  \vk^{-1}.
$$

The latter is the simpler of the two.  As
$$
B(1) = (2w'\partial + w'') R_0(\vk) \tfrac{1}{w},
$$
we need only apply the first relation in \eqref{E:w hyp} and then \eqref{w under over}.

Analogously, $B(0)$ can be handled by writing 
\begin{align*}
B(0) &= R_0(\vk) (2\partial w' - w'') R_0(\vk) \tfrac{1}{w} (-\partial^2+\vk^2) \\
&= R_0(\vk) (2\partial \tfrac{w'}{w} - \tfrac{w''}{w}) -  R_0(\vk) (2\partial w' - w'') R_0(\vk) \bigl(2\partial \tfrac{w'}{w^2} - (\tfrac{w'}{w^2})' \bigr)
\end{align*} 
and then employing \eqref{E:w hyp} and \eqref{w under over}.
\end{proof}

In connection with proving local smoothing, we shall have to deal extensively with localizing weights.  With this in mind, it is convenient to make a definitive choice of such a weight for use throughout the paper.  We select
\begin{equation}\label{E:psi defn}
\psi( x ) := \sech\bigl(\tfrac{x}{99}\bigr) \qtq{and} \psi_z( x ) := \psi( x -z ).
\end{equation}
While much about this choice is arbitrary, let us quickly mention two particular considerations. First,
\begin{equation}\label{E:del psi}
\bigl| (\partial^s_x \psi^m)(x) \bigr| \lesssim \psi^m(x)  
\end{equation}
for any pair of integers $m,s\geq 1$.  In particular, $w(x)=\psi(x)^m$ satisfies the first constraint in \eqref{E:w hyp}. Our second consideration was that the number $99$ is large enough to guarantee that $w(x)=\psi(x)^m$ also satisfies the second hypotheses of Lemma~\ref{L:w commutator} for all powers $1\leq m\leq 12$.

With the choice of cutoff made, we may now introduce the norm that will be central to all our local smoothing analysis:

\begin{definition}\label{D:LS}
For $\vk\geq 1$, we define
\begin{equation}\label{E:D:LS}
\bigl\| q \bigr\|_{\vLS} := \sup_{z\in\R} \, \bigl\| (\psi_z^6q)'' \bigr\|_{L^2_t H^{-1}_\vk ([-1,1]\times\R)} .
\end{equation}
\end{definition}

The time interval is fixed as $[-1,1]$ both for expository simplicity and because allowing for a general time interval, say $[-T,T]$, does not produce meaningfully better results.  Indeed, high-frequency wave packets can accumulate their entire local-smoothing norm in an arbitrarily short time interval and so our bounds will not improve in the limit $T\to 0$.  Conversely, for long time intervals, our arguments do not yield better results than can be obtained \emph{a posteriori} by covering $[-T,T]$ with unit intervals.

The power $6$ appearing in \eqref{E:D:LS} is chosen for its divisibility properties --- it will allow us to redistribute weights among multiple copies of $q$ without introducing fractions.  In view of such changes in powers and because derivatives of the weight will appear from commutators arising in the analysis, it is important that we prove the following:

\begin{lemma}\label{L:change}
Given \( \phi \in \Schw(\R) \),
\begin{equation}\label{E:L:change}
 \| (\phi q)'' \|_{L^2_t H^{-1}_\vk} \lesssim_\phi \| q \|_{\vLS}  + \| q \|_{L^\infty_t H^{-1}_\vk} .
\end{equation}
Analogously, for any $s\in\{0,1,2\}$, we have
\begin{equation}\label{E:L:change'}
\| \partial^s (\phi q) \|_{L^2_t H^{-1}_\vk} + \| \phi \partial^s q \|_{L^2_t H^{-1}_\vk}
	\lesssim_\phi \vk^{\frac{s-2}{3}} \Bigl\{ \| q \|_{L^\infty_t H^{-1}} + \| q \|_{\vLS}  \Bigr\} .
\end{equation}
\end{lemma}

\begin{proof}
We begin with \eqref{E:L:change}.  By construction,
$$
\int_\R \psi_z^{12}(x) \,dz  \equiv \tfrac{512}{7} \qtq{and so}
	\| \partial^2 (\phi q) \|_{H^{-1}_\vk} \leq \tfrac{7}{512}\int \| \partial^2 (\phi \psi_z^{12} q) \|_{H^{-1}_\vk} \,dz,
$$
by the triangle inequality.

On the other hand, using \eqref{elem mult}, 
\begin{align*}
\| (\phi \psi_z^{12} q)'' \|_{H^{-1}_\vk} &\lesssim \| (\phi\psi_z^6) (\psi_z^6 q)'' \|_{H^{-1}_\vk} + \| (\phi\psi_z^6)' (\psi_z^6 q)' \|_{H^{-1}_\vk} + \| (\phi\psi_z^6)'' \psi_z^6 q \|_{H^{-1}_\vk}  \\
&\lesssim \bigl\| \phi\psi_z^6 \bigr\|_{H^3} \Bigl\{ \| (\psi_z^6 q)'' \|_{H^{-1}_\vk} + \| (\psi_z^6 q)'' \|_{H^{-1}_\vk}^{1/2} \cdot \| q \|_{H^{-1}_\vk}^{1/2} + \| q \|_{H^{-1}_\vk} \Bigr\}.
\end{align*}
The estimate \eqref{E:L:change} now follows by integrating in time and noting that
\begin{equation*}
\int \bigl\| \phi\psi_z^6 \bigr\|_{H^3} \,dz \lesssim_\phi 1.
\end{equation*}

Let us now turn our attention to \eqref{E:L:change'}.  By Plancherel,
\begin{align}\label{change s}
\| \partial^s(\phi q) \|_{H^{-1}_\vk}  \lesssim \vk^{\frac{s-2}3} \Bigl\{ \| \phi q \|_{H^{-1}} + \| (\phi q)'' \|_{H^{-1}_\vk} \Bigr\}
\end{align}
holds pointwise in time.  Thus, by \eqref{elem mult} and \eqref{E:L:change} we have
\begin{align}\label{change s...}
\| \partial^s (\phi q) \|_{L^2_t H^{-1}_\vk} \lesssim_\phi \vk^{\frac{s-2}{3}} \Bigl\{ \| q \|_{L^\infty_t H^{-1}} + \| q \|_{\vLS}  \Bigr\} .
\end{align}
The remaining parts of \eqref{E:L:change'} now follow by writing $\phi q' = (\phi q)' - \phi'q $ and $\phi q'' = (\phi q)'' - 2\phi'q' - \phi''q$.
\end{proof}

The very subtlest parts of our analysis require us to exhibit certain cancellations that appear when performing two commutations in a symmetrical way.  This will be important, for example, in the proofs of Lemmas~\ref{L:2for2} and~\ref{L:cubic}.

\begin{lemma}[Double commutators]\label{L:2xC}
There is a finite collection of operators $\{A_i^{\ },A_i'\}$ satisfying
\begin{equation}\label{A:2xC}
\| A \|_{H_\vk^{-1}\rightarrow H_\vk^{1}} \lesssim 1,
\end{equation}
so that, writing $R_0=R_0(\vk)$ we have
\begin{equation}\label{R0 vs psi sym}
R_0 \psi^{12} R_0 - \psi^6 R_0^2 \psi^6 = \vk^{-2}\sum \psi^6 A_i^{\ } A_i' \psi^6 .
\end{equation}
Moreover, there is another finite collection of operators $\{A_i^{\ },A_i'\}$ satisfying \eqref{A:2xC} so that
\begin{equation}\label{psi vs q vs R0}
\psi^2 R_0 q R_0 \psi^2 = R_0 \psi^4 q R_0 - 4 \vk^2 R_0^2 [\psi^3\psi' q]' R_0^2 - 4\partial R_0^2 [\psi^3\psi' q]' \partial R_0^2  + \vk^{-2}\sum  A_i^{\ } \psi^2 q A_i'.
\end{equation}
\end{lemma}

\begin{proof}
We begin with two basic operator identities
\begin{gather}
R_0 f = f  R_0 + 2f' \partial R_0^2 + R_0 f'' ( 3\partial^2+\vk^2)R_0^2 + 2 R_0 f''' \partial R_0^2 \label{ci1}\\
\partial R_0 f - f \partial R_0 = \vk^2 R_0 f' R_0 + \partial R_0 f' \partial R_0 \label{ci2}
\end{gather}
valid for any smooth function $f$.  These can be verified by iterating Lemma~\ref{L:basic_commutator}, or by writing the corresponding integral kernels in Fourier variables.  For example, \eqref{ci2} corresponds to
$$
\bigl[\tfrac{i\xi}{\xi^2+\vk^2} - \tfrac{i\eta}{\eta^2+\vk^2}\bigr] \hat f(\xi-\eta)
	= \tfrac{i(\vk^2-\xi\eta)(\xi-\eta)}{(\xi^2+\vk^2)(\eta^2+\vk^2)} \hat f(\xi-\eta).
$$

Combining \eqref{ci1} with Lemma~\ref{L:w commutator} and \eqref{E:del psi} shows
\begin{align}
R_0 \psi^6 = \psi^6  R_0 + 12 \psi^5\psi' \partial R_0^2 + \vk^{-2} \psi^6 A \qtq{with} \| A \|_{H_\vk^{-1} \rightarrow H_\vk^1} \lesssim 1. \label{R0psi}
\end{align}
By taking adjoints (in the $L^2\to L^2$ sense), we deduce that
\begin{align}
\psi^6 R_0  = R_0 \psi^6 - 12 \partial R_0^2 \psi^5\psi'  + \vk^{-2} A^{*} \psi^6 \qtq{with} \| A^* \|_{H_\vk^{-1} \rightarrow H_\vk^1} \lesssim 1. \label{psiR0}
\end{align}

Multiplying out \eqref{R0psi} and \eqref{psiR0} brings us very close to proving \eqref{R0 vs psi sym}; the only terms not conforming to the desired representation in an obvious way are
\begin{align*}
12 \psi^5\psi' \partial R_0^3 \psi^6 - 12 \psi^6  \partial R_0^3 \psi^5\psi'
	&= 12 \psi^6 \Bigl[\tfrac{\psi'}{\psi} \partial R_0^3 -  \partial R_0^3 \tfrac{\psi'}{\psi} \Bigr] \psi^6 \\
	&= 12 \psi^6 \Bigl[ - \bigl(\tfrac{\psi'}{\psi}\bigr)' R_0^3 -  \partial \bigl[R_0^3, \tfrac{\psi'}{\psi}\bigr] \Bigr] \psi^6.
\end{align*}
To handle this, we write $[R_0^3, \phi]=R_0^2 [R_0, \phi]+ R_0 [R_0, \phi]R_0 + [R_0, \phi]R_0^2$ and apply Lemmas~\ref{L:basic_commutator} and~\ref{L:w commutator}.

To prove \eqref{psi vs q vs R0}, we begin with the following analogue of \eqref{R0psi}:
\begin{align}
R_0 \psi^2 = \psi^2  R_0 + 4 \psi\psi' \partial R_0^2 + \vk^{-2} \psi^2 A \qtq{with} \| A \|_{H_\vk^{-1} \rightarrow H_\vk^1} \lesssim 1. \label{R0psi'}
\end{align}

Let us multiply out LHS\eqref{psi vs q vs R0} using \eqref{R0psi'} and its adjoint.  When we compare the result with RHS\eqref{psi vs q vs R0}, we find ourselves left to represent
\begin{align*}
- 4 \partial R_0^2 \psi^3 \psi'  q R_0 +4 R_0 \psi^3\psi' q  \partial R_0^2
	+ 4\vk^2 R_0^2 [\psi^3\psi' q]' R_0^2 + 4\partial R_0^2 [\psi^3\psi' q]' \partial R_0^2
\end{align*}
in the form $\sum A_i \psi^2 q A_i'$.  Actually, this term is even better; \eqref{ci2} shows that it is zero.
\end{proof}

\section{Paraproducts}\label{S:para}

The initial thrust of this section is to develop basic estimates on the diagonal Green's function in terms of the local smoothing norm.  In view of the series \eqref{E:g defn}, this amounts to the discussion of the paraproducts $h_\ell$.  Later (beginning with Lemma~\ref{L:g up}), we treat certain nonlinear combinations of these terms that arise naturally in our analysis. 

Although the paraproducts that follow fit the mold of the Coifmann--Meyer theory, their symbols rapidly become so complicated as to render that approach untenable.  Rather, we employ a method that synthesizes the traditional Littlewood--Paley techniques with trace-ideal technology.

Commutation will also play a major role, because we will need to obtain localizing factors next to every copy of $q$ in the expressions \eqref{E:h defn} in order to employ the local smoothing norm.  We can then break into Littlewood--Paley pieces and deploy our basic trace-ideal estimates, such as,
\begin{align}
\bigl\|\sqrt{R_0(\vk)} f_N \sqrt{R_0(\vk)}\bigr\|_{\I_2} &\lesssim \min\{\vk^{-\frac32}N, \vk^{-\frac12}N^{-2} \} \bigl\{ \|f\|_{H^{-1}} + \bigl\|\partial^2 f\bigr\|_{H^{-1}_\vk} \bigr\}, \label{est HS}
\end{align}
which follows from Lemma \ref{L:HS}, and
\begin{align}
\bigl\|\sqrt{R_0(\vk)} f_N \sqrt{R_0(\vk)}\bigr\|_{\I_\infty} &\lesssim \min\{\vk^{-2}N^{\frac32}, \vk^{-\frac12} \} \|f\|_{H^{-1}}, \label{est op}
\end{align}
which also utilizes the Bernstein estimate $\|f_N\|_{L^\infty}\lesssim N^{3/2} \|f\|_{H^{-1}}$.

Evidently, these inequalities incorporate two distinct modes of estimation and ensure that we employ the optimal one for each pair $N,\vk$.  Experience has shown us that this optimization, as well as the proper assignment of these estimates (based on frequency comparisons), is essential in order to obtain sufficient decay (as $\vk\to\infty$) to complete the analysis in the sections that follow.

\begin{proposition}\label{P:X_l thingy}
Given $2\leq\ell\leq 4$, a collection of operators $A_1,\ldots,A_{\ell+1}$ obeying
\begin{equation}\label{op As hyp}
\| A_j \|_{H_\vk^{-1}\rightarrow H_\vk^{1}} \leq 1,
\end{equation}
a collection of Schwartz functions $\phi_1,\ldots,\phi_\ell$, and $q:[-1,1]\to B_\delta$, we write
\begin{equation}\label{Fell form}
F_\ell(t) := A_{1} \phi_{1} q(t) A_{2} \phi_{2} q(t) A_3 \hdots A_\ell \phi_{\ell} q(t) A_{\ell+1} .
\end{equation}
Then
\begin{align}\label{E:X_l thingy}
\bigl\| F_\ell \bigr\|_{L^1_t\I_1} 
\lesssim \vk^{-\frac{4}{3}-\frac{3}{2}\ell} \delta^{\ell-2} \Bigr\{ \delta^2 + \|q\|_{\vLS}^2 \Bigr\} .
\end{align}
The implicit constant depends on $\phi_1,\ldots,\phi_\ell$, but is independent of $\vk\geq 1$.
\end{proposition}

\begin{proof}
We begin by noting that \eqref{op As hyp} guarantees that
$$
A_j = \sqrt{R_0(\vk)}  B_j \sqrt{R_0(\vk)} \qtq{with} \| B_j \|_\op \lesssim 1.
$$
Decomposing all $\phi_j q$ into their Littlewood--Paley pieces and using H\"older's inequality in Schatten classes, we deduce that
\begin{align*}
\| F_\ell \|_{L^1_t \I_2} \lesssim \sum_{N_1,\ldots,N_\ell}\bigl\| \sqrt{R_0(\vk)} \bigr\|_\op^2
	\prod_{j=1}^\ell \bigl\| \sqrt{R_0(\vk)} (\phi_j q)_{N_j} \sqrt{R_0(\vk)} \bigr\|_{L^{p_j}_t\mathfrak{I}_{p_j}}
\end{align*}
provided $\sum 1/p_j =1$.  Actually, we only need a simple form of H\"older's inequality, namely, when two of the exponents are $2$ and all others are $\infty$.  However, it is important that we place the two highest frequency terms in $\I_2$, while the remaining terms we place in operator norm.  Concretely, we use
\begin{align}
\bigl\| \sqrt{R_0(\vk)} (\phi q)_N \sqrt{R_0(\vk)} \bigr\|_{L^2_t \I_2}&\lesssim_\phi \min\{\vk^{-\frac32}N, \vk^{-\frac12}N^{-2} \}\bigl[ \delta + \|q\|_{\vLS} \bigr], \label{est HS'}\\
\bigl\| \sqrt{R_0(\vk)} (\phi q)_N \sqrt{R_0(\vk)} \bigr\|_{L^\infty_t \I_\infty}   &\lesssim_\phi \min\{\vk^{-2}N^{\frac32}, \vk^{-\frac12} \} \delta, \label{est op'}
\end{align}
which follow from \eqref{est HS} and \eqref{est op} by applying \eqref{elem mult} and Lemma~\ref{L:change}.  In this way, matters reduce to controlling the sum over frequencies, which we relabel so that they are ordered.  We are led to bound
\begin{align*}
&\sum_{N_1 \geq \cdots \geq N_\ell \geq 1} \ \prod_{j=1}^2 \min \{ \vk^{-\frac{3}{2}} N_j , \vk^{-\frac{1}{2}} N_j^{-2}\}
	\prod_{j=3}^\ell \min\{ \vk^{-\frac{1}{2}}, \vk^{-2} N_j^{\frac{3}{2}} \}\\
&\lesssim \sum_{N_1 \geq N_2 } \Big( \prod_{j=1}^2 \min\{ \vk^{-\frac{3}{2}} N_j , \vk^{-\frac{1}{2}} N_j^{-2} \} \Big) \min \{ \vk^{-\frac{1}{2}}\log( 2+ \tfrac{N_2}{\vk}), \vk^{-2} N_2^{\frac{3}{2}} \}^{\ell-2}\\
&\lesssim \sum_{N_2} \min\{ \vk^{-\frac{7}{6}}, \vk^{-\frac{1}{2}} N_2^{-2}\} \min\{ \vk^{-\frac{3}{2}} N_2 , \vk^{-\frac{1}{2}} N_2^{-2} \}\min \{ \vk^{-\frac{1}{2}}\log( 2+ \tfrac{N_2}{\vk}), \vk^{-2} N_2^{\frac{3}{2}} \}^{\ell-2}  \\
&\lesssim \sum_{1\leq N_2 \leq \vk^{\frac{1}{3}}} \!\!\! \vk^{\frac{4}{3}-2\ell} N_2^{-2+\frac{3}{2}\ell}
	+ \!\!\! \sum_{\vk^{\frac{1}{3}}\leq N_2 \leq \vk} \!\!\! \vk^{3-2\ell} N_2^{-7+\frac{3}{2}\ell}
	+ \sum_{N_2\geq \vk} \vk^{-\frac{\ell}{2}} N_2^{-4}\log(2 + \tfrac{N_2}{\vk})^{\ell-2} \\
&\lesssim \vk^{\frac{2}{3}-\frac{3}{2}\ell} + \vk^{-4-\frac{\ell}{2}} \lesssim \vk^{\frac{2}{3}-\frac{3}{2}\ell}
\end{align*}
and the result follows.  Evidently, the result can be extended to $\ell\geq 5$, albeit with a different power of $\vk$.
\end{proof}

\begin{corollary}\label{C:h in L^1}
Fix $\phi\in\Schw(\R)$.  For any $q:[-1,1]\to B_\delta$ and $\vk\geq 1$,
\begin{align}
\| \phi h_3 \|_{L^1_{t,x}} &\lesssim \vk^{-5-\frac{5}{6}} \delta \bigr\{ \delta^2 + \|q\|_{\vLS}^2 \bigr\} ,\label{eq:commutator_cubic_f}\\
\sum_{\ell\geq 4} \| \phi h_\ell \|_{L^1_{t,x}} &\lesssim 	\vk^{-7-\frac{1}{3}} \delta^{2} \bigr\{ \delta^2 + \|q\|_{\vLS}^2 \bigr\}. \label{eq:commutator_quartic_f}
\end{align}
\end{corollary}

\begin{proof}
Our proof rests on the following representation 
\begin{align}\label{hell as trace}
\| \phi h_\ell \|_{L^1_{t,x}} &= \sup_{f} \  \int_{-1}^1 \Tr\Bigl\{f \phi R_0 \bigl(q R_0\bigr)^\ell \Bigr\} \,dt
\end{align}
where the supremum is over $f\in L^\infty([-1,1]\times \R)$ of unit norm.  By the same partition of unity argument exhibited in the proof of Lemma~\ref{L:change}, we see that it suffices to assume that $\phi$ is some positive power of $\psi$.  For example, if we prove \eqref{eq:commutator_cubic_f} with $\phi=\psi^3$, it then follows that it holds uniformly for translates (due to the uniformity under translation of $q$) and finally, we have
$$
\| \phi h_3 \|_{L^1_{t,x}} \leq \tfrac{512}{7} \| \phi \psi_z^9 \|_{L^1_z L^\infty_{t,x}} \| \psi_z^3 h_3  \|_{L^\infty_z L^1_{t,x}}
	\lesssim_\phi \vk^{-5-\frac{5}{6}} \delta \bigr\{ \delta^2 + \|q\|_{\vLS}^2 \bigr\} .
$$

Beginning with \eqref{eq:commutator_cubic_f}, we write
\begin{equation}\label{11:51}
\psi^3 R_0 q R_0 q R_0 q R_0 = \bigl[\psi^3 R_0 \tfrac{1}{\psi^3}\bigr] \psi q \bigl[\psi^2 R_0 \tfrac{1}{\psi^2}\bigr]  \psi q \bigl[\psi R_0 \tfrac{1}{\psi}\bigr] \psi q \bigl[ R_0 \bigr].
\end{equation}
Notice that by Lemma~\ref{L:w commutator}, each operator in square brackets is bounded as a mapping $H_\vk^{-1}\rightarrow H_\vk^{1}$.  In this way, \eqref{eq:commutator_cubic_f} follows from  Proposition~\ref{P:X_l thingy}.

Analogously, the $\ell=4$ term in \eqref{eq:commutator_quartic_f} follows from
$$
\psi^{4} R_0 (q R_0)^3 = \bigl[ \psi^{4} R_0 \tfrac{1}{\psi^{4}}\bigr] \psi q \bigl[\psi^3 R_0 \tfrac{1}{\psi^3}\bigr] \psi q \bigl[\psi^2 R_0 \tfrac{1}{\psi^2}\bigr]
		\psi q \bigl[\psi R_0 \tfrac{1}{\psi}\bigr] \psi q \bigl[ R_0 \bigr].
$$

It remains to treat $\ell>4$, which we shall do by reducing consideration to the case $\ell=4$.  Cycling the trace in \eqref{hell as trace} yields
\begin{align}\label{h ell as trace'}
\!\!\Tr\Bigl\{\!\:\!f \psi^4 R_0 \bigl(q R_0\bigr)^\ell \Bigr\} &=  \Tr\Bigl\{\!\:\!\bigl(\!\sqrt{R_0} q \sqrt{R_0}\,\bigr)^{2} \!\sqrt{R_0} \,\psi^{2} f \psi^{2} \!\sqrt{R_0}
	\bigl(\!\sqrt{R_0} q \sqrt{R_0}\,\bigr)^{l-2} \Bigr\}.
\end{align}
When $\ell=4$, the operators to the left and right of $f$ are adjoints of one another.  Therefore, the trace is monotone in $f$ and so maximized by taking $f\equiv 1$.  Thus,
\begin{align*}
\| \psi^4 h_4 \|_{L^1_{t,x}} &= \bigl\| \bigl(\sqrt{R_0} q \sqrt{R_0}\bigr)^{2} \sqrt{R_0} \, \psi^{2} \bigr\|_{L^2_t \I_2}^2 .
\end{align*}

Returning to \eqref{h ell as trace'} for general $\ell\geq 5$, we then deduce that
\begin{align*}
\| \psi^4 h_\ell \|_{L^1_{t,x}} &\leq  \bigl\| \bigl(\sqrt{R_0} q \sqrt{R_0}\bigr)^{2} \sqrt{R_0}\, \psi^{2} \bigr\|_{L^2_t\I_2}^2 \bigl\| \sqrt{R_0} q \sqrt{R_0} \bigr\|_{L^\infty_t \mathfrak{I}_\infty}^{\ell-4} \leq \| \psi^{4} h_4 \|_{L^1_{t,x}}  \delta^{\ell-4}.
\end{align*}
As $\delta\ll 1$, this can be summed to yield \eqref{eq:commutator_quartic_f}.
\end{proof}

Our next proposition and corollary are close analogues of the previous ones; however, we now estimate the $H^1_\vk$ norm and (more importantly) only use one copy of the local smoothing norm to do so.

\begin{proposition}\label{P:Y_l thingy}
Given $\ell\geq 2$, a collection of operators $A_1,\ldots,A_{\ell+1}$ obeying \eqref{op As hyp}, a collection of Schwartz functions $\phi_1,\ldots,\phi_\ell$, and $f,q:[-1,1]\to B_\delta$, we write
$$
F_\ell(t) := \sqrt{R_0(\vk)} \, f(t) A_{1} \phi_{1} q(t) A_{2} \phi_{2} q(t) A_3 \hdots A_\ell \phi_{\ell} q(t) \sqrt{R_0(\vk)} .
$$
Then 
\begin{align}\label{E:Y_l thingy}
\| F_\ell \|_{L^1_t \I_1} \lesssim  \vk^{-\frac\ell2-\frac{13}6} \delta^{\ell-1} \bigr\{ \delta + \| q\|_{\vLS} \bigr\}  \| f\|_{L^2_t H^{-1}_\vk}.
\end{align}
The implicit constant is independent of $\vk\geq 1$.
\end{proposition} 

\begin{proof}
As in the proof of the previous proposition, we proceed by H\"older's inequality and decomposing into Littlewood--Paley pieces.  This time, we apply \eqref{est HS'} only to  the highest frequency term, and \eqref{est op'} to the remainder.  Regarding $f$, we simply use Lemma~\ref{L:HS}:
\begin{equation*}
\bigl\|  \sqrt{R_0(\vk)}  f \sqrt{R_0(\vk)} \bigr\|_{L^2_t \I_\infty} 
	\leq \bigl\|  \sqrt{R_0(\vk)}  f \sqrt{R_0(\vk)} \bigr\|_{L^2_t \I_2}
			\lesssim \vk^{-\frac12} \| f \|_{L^2_t H^{-1}_\vk}.
\end{equation*}

In this way, we are once again led to a sum over ordered frequencies:
\begin{align*}
&\sum_{N_1 \geq \cdots \geq N_\ell} \vk^{-\frac{1}{2}} \min \{ \vk^{-\frac{3}{2}} N_1, \vk^{-\frac{1}{2}} N_1^{-2} \}
	\prod_{j=2}^\ell \min\{ \vk^{-2} N_j^{\frac{3}{2}}, \vk^{-\frac{1}{2}} \} \\
&\lesssim  \sum_{N_1}\vk^{-\frac{1}{2}} \min \{ \vk^{-\frac{3}{2}} N_1, \vk^{-\frac{1}{2}} N_1^{-2} \}
	\min\{ \vk^{-2} N_1^{\frac{3}{2}}, \vk^{-\frac{1}{2}}\log( 2 + \tfrac{N_1}{\vk}) \}^{\ell-1}  \\
&\lesssim \vk^{-\frac{3\ell}2-\frac16} + \vk^{-\frac\ell2-\frac52} \lesssim \vk^{-\frac\ell2-\frac{13}6} .
\end{align*}
Once again, the last step involves breaking the sum into three pieces. Note that for $\ell\geq 3$, we can obtain a slightly better power, namely, $\vk^{-\frac{\ell+5}2}$.
\end{proof}

\begin{corollary}\label{C:h in H^1}
Fix $\phi\in\Schw(\R)$ and $2\leq m\leq 12$.  Then for $\delta$ sufficiently small,
\begin{equation}\label{E:h in H^1}
\sum_{\ell\geq m} \| \phi h_\ell \|_{L^2_t H^1_\vk} \lesssim \vk^{-\frac{m}2-\frac{13}6} \delta^{m-1} \bigr\{ \delta + \|q\|_{\vLS} \bigr\},
\end{equation}
uniformly for $q:[-1,1]\to B_\delta$ and  $\vk\geq 1$.
\end{corollary}

\begin{proof}
We begin by proving that for each $f:[-1,1]\to H^{-1}$ and $\ell\leq 12$,
\begin{equation}\label{E:h in H^1'}
\bigl\| \sqrt{R_0} f \phi (R_0 q)^\ell \sqrt{R_0} \,\bigr\|_{L^1_t \I_1}
	\lesssim \vk^{-\frac\ell2-\frac{13}6} \delta^{\ell-1} \bigr\{ \delta + \|q\|_{\vLS} \bigr\} \| f\|_{L^2_t H^{-1}_\vk} .
\end{equation}

Imagine first that $\phi=\psi_z^{\ell}$. Mimicking the proof of Corollary~\ref{C:h in L^1}, we write
$$
\psi_z^{\ell} (R_0 q)^\ell  = [\psi_z^\ell R_0 \psi_z^{-\ell}] \psi_z q [\psi_z^{\ell-1}R_0\psi_z^{1-\ell}]\psi_z q \cdots
	[\psi_z R_0\psi_z^{-1}]\psi_z q 
$$
and note that Lemma~\ref{L:w commutator} applies to all operators in square brackets.   In view of translation invariance, it then follows from Proposition~\ref{P:Y_l thingy} that
$$
\| \sqrt{R_0} f \psi_z^\ell (R_0 q)^\ell \sqrt{R_0} \|_{L^\infty_z L^1_t \I^1_\vk}
	\lesssim \vk^{-\frac\ell2-\frac{13}6} \delta^{\ell-1} \bigr\{ \delta + \|q\|_{\vLS} \bigr\} \| f\|_{L^2_t H^{-1}_\vk}.
$$
This then implies \eqref{E:h in H^1'} by the partition of unity argument used earlier:
\begin{align*}
\| \sqrt{R_0} f \phi (R_0 q)^\ell \sqrt{R_0} \|_{L^1_t \I_1} &\lesssim \int_\R \| \sqrt{R_0} f \psi_z^{2\ell} \phi (R_0 q)^\ell \sqrt{R_0} \|_{L^1_t \I_1} \,dz \\
	&\lesssim \vk^{-\frac\ell2-\frac{13}6} \delta^{\ell-1} \bigr\{ \delta + \|q\|_{\vLS} \bigr\}  \| \phi\psi^\ell_z f \|_{L^1_z L^2_t H^{-1}_\vk}  \\
	&\lesssim \vk^{-\frac\ell2-\frac{13}6} \delta^{\ell-1} \bigr\{ \delta + \|q\|_{\vLS} \bigr\}   \| \psi_z^\ell\phi \|_{L^1_z H^1} \| f\|_{L^2_t H^{-1}_\vk}.
\end{align*}

While the limit $\ell\leq 12$ applied in \eqref{E:h in H^1'} is rather arbitrary, it seems untenable to keep track of the dependence of the implicit constant on $\ell$; thus, it cannot be applied directly to control the tail of the series in \eqref{E:h in H^1}.

Given $m\leq 12$ and $\ell\geq m$, we deduce from \eqref{E:h in H^1'} and \eqref{E:L:HS} that
\begin{align*}
\iint \phi(x) h_\ell(t,x;\kappa) f(t,x)\,dx\,dt
	&\leq \bigl\| \sqrt{R_0} f \phi (R_0 q)^m \sqrt{R_0} \bigr\|_{L^1_t \I_1} \bigl\| \sqrt{R_0} q \sqrt{R_0}\bigr\|_{L^\infty_t \I_\infty}^{\ell-m} \\
&\lesssim \vk^{-\frac\ell2-\frac{13}6} \delta^{\ell-1} \bigr\{ \delta + \|q\|_{\vLS} \bigr\} \| f\|_{L^2_t H^{-1}_\vk}, 
\end{align*}
where the implicit constant depends on $m$ and $\phi$, but not on $\ell$.  The lemma now follows by duality.
\end{proof}

Although not strictly a paraproduct, we record here several key estimates for $h_1$ similar in nature to that of Corollary~\ref{C:h in H^1}.  

\begin{proposition}\label{P:h1}
Given $\phi\in\Schw(\R)$ and $q:[-1,1]\to B_\delta$,
\begin{align}\label{E:h1 with s}
\bigl\| \phi h_1^{(s)}  \bigr\|_{L^2_{t,x}} &\lesssim \vk^{-\frac{8-s}3} \bigl\{ \delta + \| q \|_{\vLS} \bigr\}
\end{align}
for each $s\in\{0,1,2\};$ moreover,
\begin{align}\label{E:h1 other}
\bigl\| \phi h_1 ''  \bigr\|_{L^2_t H^{1}_\vk} &\lesssim \vk^{-1} \bigl\{ \delta + \| q \|_{\vLS} \bigr\} \qtq{and}
\bigl\| \phi h_1 '  \bigr\|_{L^4_{t,x}}^2   \lesssim \vk^{-\frac72} \delta \bigl\{ \delta + \| q \|_{\vLS} \bigr\}.
\end{align} 
\end{proposition}

\begin{proof}
By \eqref{h1 nice}, $-\vk\,\psi h_1^{(s)} =  [ \psi R_0(2\vk)\tfrac1{\psi} ] \psi q^{(s)}$. Thus, by Lemmas~\ref{L:w commutator} and~\ref{L:change},
\begin{align*}
\vk^2 \bigl\| \psi h_1^{(s)} \bigr\|_{L^2_{t,x}} \lesssim \vk \bigl\| \psi h_1^{(s)} \bigr\|_{L^2_t H^1_\vk}
	&\lesssim  \|  \psi q^{(s)} \|_{L^2_t H^{-1}_\vk} \lesssim \vk^{\frac{s-2}{3}} \bigl\{ \delta + \| q \|_{\vLS} \bigr\},
\end{align*}
which proves \eqref{E:h1 with s} and the first inequality in \eqref{E:h1 other}, provided $\phi=\psi$.  The case of general $\phi$ then follows by the partition of unity argument exhibited already in the proof of Lemma~\ref{L:change}.

Turning now to the second inequality in \eqref{E:h1 other}, we write $\phi h_1 ' = (\phi h_1) ' - \phi' h_1$ and apply the Gagliardo--Nirenberg and H\"older inequalities:
\begin{align*}
\| \phi h_1 ' \|_{L^4_{t,x}}^2 \lesssim \| (\phi h_1)'' \|_{L^2_{t,x}} \| \phi h_1 \|_{L^\infty_{t,x}}
	+ \| (\phi' h_1) ' \|_{L^2_{t,x}} \| \phi' h_1 \|_{L^\infty_{t,x}}.
\end{align*}
The result now follows from \eqref{E:h1 with s} and \eqref{E:h in Linfty}.
\end{proof}

The diagonal Green's function appears in the \emph{denominator} in both  \eqref{E:rho dot 5} and \eqref{E:rho dot k}.  Our next lemma provides us with an efficient means of handling this situation in terms of the results we have already developed. 

\begin{lemma}\label{L:g up}
Fix $\phi\in\Schw(\R)$.  For $q:[-1,1]\to B_\delta$ and $\delta$ sufficiently small, we have
\begin{align}\label{E:L:g up} 
\bigl\| \phi \bigl[\tfrac{1}{g} - 2\vk + \tfrac{4\vk^2 h_1}{ 1+2\vk h_1 }\bigr] \bigr\|_{L^2_t H^1_\vk}
	&\lesssim_\phi \vk^{-\frac76} \delta  \bigr\{ \delta + \|q\|_{\vLS} \bigr\}.
\end{align}
\end{lemma}

\begin{proof}
There is no fear of division by zero here: From \eqref{E:h in H1} we see that
$$
\| 2\vk h_1 \|_{L^\infty_t H^1_\vk} \lesssim \delta \qtq{and so} \bigl\| \tfrac{2\vk h_1}{ 1+2\vk h_1 } \bigr\|_{L^\infty_t H^1_\vk} \lesssim \delta,
$$
provided $\delta$ is sufficiently small.  Combining this with the identity
\begin{align}
\tfrac{1}{g} - 2\vk + \tfrac{4\vk^2 h_1}{ 1+2\vk h_1 } = - \tfrac{2\vk}{g} \bigl[ 1 - \tfrac{2\vk h_1}{(1+2\vk h_1)}\bigr] [g-\tfrac1{2\vk}-h_1],
\end{align}
we see that the lemma then follows from Corollary~\ref{C:h in H^1}.
\end{proof}

In the remainder of this section, we prove two further paraproduct bounds on nonlinear combinations of the $h_\ell$ that arise later in the analysis.  While these terms can be bounded by combining the estimates proved already, this does not yield sufficient decay.  It is essential for what follows that the exponent of $\vk^{-1}$ in \eqref{E:hhh} exceeds $7$ and that in \eqref{E:h1h3} exceeds $8$.  This does not seem to be possible other than by treating these paraproducts holistically.

\begin{lemma}\label{L:hhh}
Fix $A\in\{\Id,\, \partial^2R_0(2\vk),\, \vk^2R_0(2\vk)\}$.  Then
\begin{align}\label{E:hhh}
\bigl\| \psi^{12} h_1 A [h_1' h_1']  \bigr\|_{L^1_{t,x}} + \bigl\| \psi^{12} h_1' A [h_1' h_1^{\ }]  \bigr\|_{L^1_{t,x}}
	\lesssim \vk^{-7-\frac16} \delta \bigl\{ \delta^2 + \| q \|_{\vLS}^2 \bigr\},
\end{align}
uniformly for $q:[-1,1]\to B_\delta$.
\end{lemma}

\begin{proof}
Schur's test shows that $\psi^8 A \frac{1}{\psi^8}$ is $L^p$-bounded for every $1\leq p \leq \infty$, for each choice of $A$, and uniformly for $\vk\geq 1$.

In a similar vein, we note that
\begin{align*}
\psi^4 h_1 = - \vk^{-1} \psi^4 R_0 \tfrac{1}{\psi^4} [\psi^4 q] \qtq{and}
	\psi^4 h_1' = - \vk^{-1} \psi^4 R_0 \tfrac{1}{\psi^4} [(\psi^4 q)' - (\psi^4)' q].
\end{align*}

Combining these observations with Lemma~\ref{L:w commutator} and Bernstein inequalities, we see that it suffices to show
\begin{align*}
\sum_{N_1\geq N_2\geq N_3 } \!\!\! \frac{N_1 N_2}{\vk^3}  \| R_0(\phi_1 q)_{N_1} \|_{L^2_{t,x}}  \| R_0(\phi_2 q)_{N_2} \|_{L^2_{t,x}}
	\| R_0 (\phi_3 q)_{N_3} \|_{L^\infty_{t,x}} \lesssim \text{RHS\eqref{E:hhh}},
\end{align*}
for any trio of Schwartz functions $\phi_i\in\{\psi^4, \partial_x\psi^4\}$.

We sum over $N_3$ first using 
\begin{align*}
\sum_{N_3\leq N_2} \| R_0 (\phi_3 q)_{N_3} \|_{L^\infty_{t,x}} \lesssim \sum_{N_3\leq N_2} \tfrac{N_3^{3/2}}{N_3^2+\vk^2} \| q \|_{L^\infty_t H^{-1}}
	\lesssim \vk^{-2} \delta N_2^{3/2} \lesssim \vk^{-2} \delta N_1^{3/4} N_2^{3/4}.
\end{align*}
To complete the proof, we substitute this into the above, and use 
\begin{align*}
\sum_{N}  N^{\frac74}\| R_0 (\phi_3 q)_{N} \|_{L^2_{t,x}} \lesssim \sum_{N \leq \vk^{1/3}} \!\! \tfrac{N^{\frac{11}4}}{\vk^2}\delta
	\, + \!\!\!\sum_{N \geq \vk^{1/3}} \!\! \tfrac{N^{\frac74}}{\vk N^2}\|q\|_{\vLS} \lesssim \vk^{-\frac{13}{12}} \bigl\{ \delta^2 + \| q \|_{\vLS}^2 \bigr\}
\end{align*}
to sum over $N_1$ and $N_2$ (we may now neglect the ordering).
\end{proof}

\begin{lemma}\label{L:h1h3}
For $q:[-1,1]\to B_\delta$ and $\delta$, we have
\begin{align}\label{E:h1h3}
\| \psi^{12} h_1 h_3 \|_{L^1_{t,x}} \lesssim \vk^{-8 - \frac13} \delta^2 \bigl\{ \delta^2 + \| q \|_{\LS}^2 \bigr\} .
\end{align}
\end{lemma}

\begin{proof}
As in the proof of Corollary~\ref{C:h in L^1}, we employ duality, writing LHS\eqref{E:h1h3} as
\begin{align}\label{6:59}
\sup_f \int_{-1}^1 \Tr\Bigl\{ \sqrt{R_0}\, \psi^{3} h_1 f \bigl[\psi^9 R_0 \tfrac{1}{\psi^9}\bigr] \psi^3 q \bigl[\psi^6 R_0 \tfrac{1}{\psi^6}\bigr]  \psi^3 q \bigl[\psi^3 R_0 \tfrac{1}{\psi^3}\bigr] \psi^3 q \sqrt{R_0}\,\Bigr\} \,dt
\end{align}
where the supremum is over $f\in L^\infty([-1,1]\times \R)$ of unit norm.

To proceed, we decompose $\psi^3 h_1 = \sum_N u_N$ with
$$
u_N := - \vk^{-1} \psi^3 R_0(2\vk) \tfrac{1}{\psi^3} [ (\psi^3 q)_{N} ].
$$
Using Lemmas~\ref{L:w commutator} and~\ref{L:HS}, we deduce the following analogues of \eqref{est HS'} and \eqref{est op'}:
\begin{gather*}
\bigl\| \sqrt{R_0}\,u_N f\sqrt{R_0}\,\bigr\|_{L^2_t \I_2} \lesssim \vk^{-\frac32} \| u_N \|_{L^2_{t,x}}
	\lesssim \vk^{-3} \min\{\vk^{-\frac32}N, \vk^{-\frac12}N^{-2} \}\bigl[ \delta + \|q\|_{\vLS} \bigr] \\
\bigl\| \sqrt{R_0}\,u_N f \sqrt{R_0}\,\bigr\|_{L^\infty_t \I_\infty}
	\lesssim \vk^{-2} \| u_N\|_{L^\infty_{t,x}} \lesssim \vk^{-3} \min\{ \vk^{-2} N^{\frac32}, \vk^{-\frac12} \} \delta . 
\end{gather*}

Returning to \eqref{6:59}, we divide each copy of $\psi^3 q$ into its Littlewood--Paley pieces, yielding a sum over four frequencies.  As in the proof of Corollary~\ref{C:h in L^1}, we now apply H\"older's inequality in trace ideals, placing the two highest frequencies pieces in $L^2_t \I_2$ and the remainder in $L^\infty\I_\infty$.   But for the prefactor $\vk^{-3}$, the resulting sum is exactly that appearing in the proof of Proposition~\ref{P:X_l thingy} when $\ell=4$ and so the result follows from the computations given there.
\end{proof}

\section{Local smoothing}\label{S:LS}

This section is primarily devoted to the proof of the following: 

\begin{proposition}\label{P:LS5}
Fix $\delta$ sufficiently small.  For any initial data $q(0) \in \mathcal{S}(\R) \cap B_\delta$, the corresponding solution $q(t)$ to \eqref{5th} satisfies
\begin{equation}\label{E:P:LS5}
\| q \|_{\vLS}^2 \lesssim  \| q(0) \|_{H^{-1}_\vk}^2 + \vk^{-\frac16} \delta^2
\end{equation}
uniformly for $\vk\geq 1$.  For the definition of $\vLS$, see \eqref{E:D:LS}.
\end{proposition}

Later in Corollary~\ref{C:LS5}, we will observe that this yields \eqref{E:T:LS} as an a priori bound for initial data in $\mathcal{S}(\R) \cap B_\delta$.  This can then be extended to large Schwartz data by employing the scaling \eqref{E:scaling}.  

Following the pattern introduced already by Kato in \cite{MR0759907}, we will be proving the local smoothing estimate by localizing the microscopic conservation law \eqref{mscl}.  While the dominant term in the current will reproduce LHS\eqref{E:P:LS5} nicely, a wide variety of the errors arising in our analysis must be bootstrapped.  This necessitates a smallness parameter.  For terms cubic and higher in $q$, the parameter $\delta$ could be used; however, there are also quadratic error terms (arising from commutators) and so an alternate source of smallness is required.  This smallness will be obtained by requiring $\vk\geq \vk_0$ for some absolute $\vk_0\gg1$.  This restriction can be removed at the end by employing the equivalence of norms (individually for $H^{-1}_\vk$ and $\vLS$) as $\vk\geq 1$ varies over bounded intervals.

To proceed, we let $\Psi(x)=\int_x^\infty \psi(x')^{12}\,dx'$ and so observe that
\begin{equation}\label{E:int micro}
\int [\rho(1,x) - \rho(-1,x)] \Psi(x)\,dx = - \int_{-1}^1 \int_{-\infty}^\infty j_\text{5th}(t,x) \psi(x)^{12}\,dx\,dt,
\end{equation}
where $j_\text{5th}$ is defined in \eqref{E:rho dot 5}.
Note that we have omitted the translation parameter $z$ that appears in the definition of $\vLS$ to avoid cumbersome notation; it will be recovered at the end by employing the translation symmetry of solutions.

Let us begin our analysis by identifying two key constituents of the current $j_\text{5th}$:
\begin{align*}
j_1 &:= 4\vk^2 [ 3 q^2  - 16\vk^5 h_2(\vk) ] -  4\vk^2 R_0(2\varkappa) \big[ - 5 (q^2)'' + 5 (q')^2\big] + 8\vk^4 h_1(\vk) h_1^{(4)}(\vk)  \\
j_2 &:= -64\vk^7 h_3(\vk) - 10 q^3 .
\end{align*}
The first part collects all quadratic terms and will be the source of the coercivity we seek.  The second current represents the  dominant cubic error term; extensive analysis will be required in order to illustrate the main cancellation therein and then control the remainder.

The next lemma controls the contributions of the remaining parts of $j_\text{5th}$ in a satisfactory manner.

\begin{lemma}\label{LS:j3}
For $q:[-1,1]\to B_\delta\cap\Schw(\R)$ and $\delta$ sufficiently small,
\begin{gather*}
\biggl| \int_{-1}^1 \int_{-\infty}^\infty [j_\text{\rm 5th}-j_1-j_2](t,x) \psi(x)^{12}\,dx\,dt \biggr|
	\lesssim \vk^{-\frac16} \bigr\{ \delta^2 + \|q\|_{\vLS}^2 \bigr\} .
\end{gather*}
\end{lemma}

\begin{proof}
Let us begin by recalling the definition of $j_\text{5th}$:
\begin{align*}
j_\text{5th} &= - \tfrac{2\vk}{g(\vk)} \bigl\{ 16\vk^5 \bigl[g(\vk)-\tfrac{1}{2\vk}\bigr] + 4 \vk^2 q + q'' - 3 q^2 \bigr\} \\
&\quad -  4\vk^2 R_0(2\varkappa) \big[q^{(4)} - 5 (q^2)'' + 5 (q')^2 + 10 q^3\big].
\end{align*}
To proceed, we expand out $g(\vk)$ using the series \eqref{E:g defn} and discuss the main calculations attendant to each value of $\ell$ in succession.  In this way, we will `discover' the terms in $j_1$ and $j_2$ as we progress and identify four key quantities to be estimated, which we label $E_1,\ldots,E_4$.  As there is no other energy parameter under consideration, we omit the argument $\vk$ from both $g$ and $h_\ell$ below. 

We start with $\ell=1$.  From \eqref{E:h1 exp} and \eqref{h1 nice},
$$
\tfrac{2\vk}{g} [  - 4 \vk^2 q - q'' - 16\vk^5 h_1 ] -  4\vk^2 R_0(2\vk) q^{(4)} = \vk \bigl[4\vk^2 - \tfrac{2\vk}{g} \bigr] h_1^{(4)}.
$$
Removing the quadratic term $8\vk^4 h_1 h_1^{(4)}$ which appears in $j_1$, we are left to estimate
\begin{align}\label{residue 1}
E_1 := \iint 2\vk^2 \bigl[2\vk -\tfrac{1}{g} - 4\vk^2 h_1 \bigr] h_1^{(4)} \psi^{12}\,dx\,dt .
\end{align}

We turn now to $\ell=2$ and consider the contribution of 
$$
\tfrac{2\vk}{g(\vk)} [ 3 q^2  - 16\vk^5 h_2 ] -  4\vk^2 R_0(2\varkappa) \big[ - 5 (q^2)'' + 5 (q')^2\big].
$$
Setting aside those terms which appear in $j_1$, we are left to estimate
\begin{align}\label{residue 2}
E_2 := \iint 2\vk \bigl[\tfrac{1}{g}  - 2\vk \bigr] [ 3 q^2  - 16\vk^5 h_2 ] \psi^{12}\,dx\,dt .
\end{align}

For $\ell=3$, we consider the combination
$$
\tfrac{2\vk}{g(\vk)} [ - 16\vk^5 h_3 ] -  4\vk^2 R_0(2\varkappa) \big[ 10 q^3\big].
$$
Setting aside the contribution of $j_2$, we are left to estimate
\begin{align}\label{residue 3a}
E_{3a} :=  - \iint 32\vk^6 \bigl[\tfrac{1}{g}  - 2\vk \bigr] h_3  \psi^{12}\,dx\,dt
\end{align}
and
\begin{align}\label{residue 3b}
E_{3b} :=  10 \iint \Bigl(\bigl[\Id -  4\vk^2 R_0(2\varkappa) \bigr] q^3 \Bigr) \psi^{12}\,dx\,dt .
\end{align}

Finally, we consider those $\ell\geq 4$.  This yields
$$
E_4:= \sum_{\ell\geq 4} \iint \tfrac{2\vk}{g(\vk)} [ - 16\vk^5 h_\ell ] \psi^{12}\,dx\,dt,
$$
which is easily estimated:  By \eqref{E:h in Linfty} and Corollary~\ref{C:h in L^1},
\begin{align*}
|E_4| \lesssim \vk^6 \bigl\| \tfrac{1}{g} \bigr\|_{L^\infty} \sum_{\ell\geq 4}  \bigl\| \psi^{12} h_\ell \bigr\|_{L^1_{t,x}}
	\lesssim  \vk^{-\frac{1}{3}} \delta^{2} \bigr\{ \delta^2 + \|q\|_{\vLS}^2 \bigr\}.
\end{align*}

We now turn our attention to estimating the remaining three error terms, beginning with $E_1$.  From \eqref{E:L:g up} and Proposition~\ref{P:h1},
\begin{align}\label{residue 1a}
\biggl| \iint 2\vk^2 \bigl[\tfrac{1}{g} - 2\vk + \tfrac{4\vk^2 h_1}{ 1+2\vk h_1 }\bigr] h_1^{(4)} \psi^{12}\,dx\,dt \biggr|
	\lesssim \vk^{-\frac16} \bigr\{ \delta^2 + \|q\|_{\vLS}^2 \bigr\} .
\end{align}
Thus the estimation of $E_1$ is reduced to controlling
\begin{align}\label{residue 1b}
E_{1a} = \iint  \tfrac{16\vk^5}{ 1+2\vk h_1 } h_1^2 h_1^{(4)} \psi^{12}\,dx\,dt.
\end{align}
Integrating by parts twice and employing \eqref{E:h in Linfty} shows
\begin{align*}
|E_{1a}|  &\lesssim \vk^{7/2} \| \psi h_1 '' \|_{L^2_{t,x}}^2 + \vk^5 \| \psi h_1 '' \|_{L^2_{t,x}} \| \psi h_1' \|_{L^4_{t,x}}^2 \\
&\quad +  \vk^{7/2} \| \psi h_1 '' \|_{L^2_{t,x}} \| \psi h_1' \|_{L^2_{t,x}}
	+ \vk^{7/2}  \| \psi h_1 '' \|_{L^2_{t,x}} \| \psi h_1 \|_{L^2_{t,x}}.
\end{align*}
In view of the results of Proposition~\ref{P:h1}, it follows that
\begin{align*}
|E_{1a}|  &\lesssim \vk^{-\frac12} \bigr\{ \delta^2 + \|q\|_{\vLS}^2 \bigr\} ,
\end{align*}
which is acceptable.

We now turn our attention to $E_2$, which we break into three pieces:
\begin{align*}
E_{2a} &= \iint 2\vk \bigl[\tfrac{1}{g}  - 2\vk \bigr] [ 3 \vk^2 (h_1'')^2 ] \psi^{12}\,dx\,dt \\
E_{2b} &= \iint 2\vk \bigl[\tfrac{1}{g}  - 2\vk + \tfrac{4\vk^2 h_1}{ 1+2\vk h_1 } \bigr] [ 3 q^2  - 16\vk^5 h_2 - 3 \vk^2 (h_1'')^2] \psi^{12}\,dx\,dt \\
E_{2c} &= \iint 2\vk \bigl[- \tfrac{4\vk^2 h_1}{ 1+2\vk h_1 } \bigr] [ 3 q^2  - 16\vk^5 h_2 - 3 \vk^2 (h_1'')^2] \psi^{12}\,dx\,dt  .
\end{align*}
The first of these is easily estimated via \eqref{E:h in Linfty} and Proposition~\ref{P:h1}:
\begin{align*}
|E_{2a}| &\lesssim \vk^3 \bigl\| \tfrac{1}{g}  - 2\vk \bigr\|_{L^\infty_{t,x}}  \bigl\| \psi h_1 ''  \bigr\|_{L^2_{t,x}}^2 
	\lesssim \vk^{-\frac12} \bigr\{ \delta^2 + \|q\|_{\vLS}^2 \bigr\} .
\end{align*}
On the other hand, from \eqref{E:h2 exp} we have
\begin{equation}\label{hex again}
\begin{aligned}
3q^2 -  16 \vk^5 h_2 - 3 \vk^2 (h_1'')^2&= - 4 \vk^4 \bigl[ 10 h_1 h_1''  + 5 (h_1')^2 \bigr] \\
&\quad - 4 \vk^4 \partial_x^2 R_0(2\vk) \bigl[ 4h_1h_1'' + 5 (h_1')^2 \bigr].
\end{aligned}
\end{equation}
Using Proposition~\ref{P:h1}, we may then deduce that
\begin{align*}
\| 3q^2 -  16 \vk^5 h_2 - 3 \vk^2 (h_1'')^2 \|_{L^2_{t,x}} &\lesssim \vk^4 \| \psi h_1 '' \|_{L^2_{t,x}} \| \psi h_1 \|_{L^\infty_{t,x}}
	+  \vk^4 \| \psi h_1' \|_{L^4_{t,x}}^2 \\
&\lesssim \vk^{\frac1{2}}\bigr\{ \delta + \|q\|_{\vLS} \bigr\} .
\end{align*}
Combining this with \eqref{E:L:g up}, it follows that
\begin{align*}
|E_{2b}| &\lesssim \vk^{-\frac23} \bigr\{ \delta^2 + \|q\|_{\vLS}^2 \bigr\} .
\end{align*}

In view of \eqref{hex again}, integration by parts allows one to reduce $E_{2c}$ to terms covered by Lemma~\ref{L:hhh}.  Thus
\begin{align*}
|E_{2c}| &\lesssim \vk^{-\frac16} \delta \bigl\{ \delta^2 + \| q \|_{\vLS}^2 \bigr\} .
\end{align*}

It remains only to consider $E_{3a}$ and $E_{3b}$.  Employing \eqref{E:h in Linfty}, Corollary~\ref{C:h in H^1}, and Lemmas~\ref{L:g up} and~\ref{L:h1h3}, we find
\begin{align*}
| E_{3a} | &\lesssim \vk^6 \bigl\| \psi^3 \bigl[\tfrac{1}{g} - 2\vk + \tfrac{4\vk^2 h_1}{ 1+2\vk h_1 }\bigr] \bigr\|_{L^2_{t,x}} \bigl\| \psi^9 h_3 \bigr\|_{L^2_{t,x}}
		+ \vk^8 \| \psi^{12} h_1 h_3 \|_{L^1_{t,x}} \\
	&\lesssim \vk^{-\frac13} \delta  \bigr\{ \delta + \|q\|_{\vLS} \bigr\}.
\end{align*}

To estimate $E_{3b}$, we first write $\Id -  4\vk^2 R_0(2\varkappa)=-\partial^2R_0(2\varkappa)$ and integrate by parts.  Next employing Lemma~\ref{L:w commutator} and the algebra property of $H^1$, we deduce that
\begin{equation*}
| E_{3b} | \lesssim  \vk^{-2} \|  \psi^2 q^3 \|_{L^1_t H^{-1}} \lesssim \vk^{-2} \delta  \|  \psi q \|_{L^2_t H^{1}}^2 .
\end{equation*}
But then by Plancherel and Lemma~\ref{L:change},
\begin{equation*}
| E_{3b} | \lesssim  \vk^{-\frac23} \delta \bigl\{  \| \psi q \|_{L^\infty_t H^{-1}}^2 + \| (\psi q)'' \|_{L^2_t H^{-1}_\vk}^2 \bigr\}
	\lesssim \vk^{-\frac23} \delta \bigl\{ \delta^2 + \| q \|_{\vLS}^2 \bigr\}. \qedhere
\end{equation*}
\end{proof}

We now demonstrate the coercivity of the quadratic current $j_1$.  It was the knowledge of precisely the nature of this coercivity that led to the definition of the local-smoothing norm back in \eqref{E:D:LS}.  For the purposes of this section, it is \eqref{2for2'} below that is important; however, we also isolate in \eqref{2for2} a part of the argument that will be useful in Section~\ref{S:dLS}.

\begin{lemma}[Quadratic current]\label{L:2for2}
Given $q:[-1,1]\to B_\delta\cap\Schw(\R)$, let
\begin{align*}
I(\vk):= \int_{-1}^1 \int_{-\infty}^\infty  \bigl\{64\vk^7h_2(\vk) - 12 \vk^2 q^2 + 5 (q')^2 - 5 (q^2)'' -  8\vk^4 [h_1''(\vk)]^2 \bigr\} \psi^{12} \,dx\,dt.
\end{align*}
Then
\begin{align}\label{2for2}
\| (\psi^6 q )''\|_{L^2_t H^{-1}_\vk([-1,1]\times\R)}^2 \lesssim  I(\vk) + \vk^{-\frac13} \bigr\{ \delta^2 + \|q\|_{\vLS}^2 \bigr\}
\end{align}
and analogously,
\begin{align}\label{2for2'}
\| (\psi^6 q )''\|_{L^2_t H^{-1}_\vk([-1,1]\times\R) }^2 \lesssim  - \int_{-1}^1 \int_{-\infty}^\infty   j_1 \psi^{12} \,dx\,dt  + \vk^{-\frac13} \bigr\{ \delta^2 + \|q\|_{\vLS}^2 \bigr\}.
\end{align}
\end{lemma}

\begin{proof}
The greater part of the work here is demonstrating \eqref{2for2}; this is were we focus our attention first.  At the end, we will see how to deduce \eqref{2for2'} from this.  
Throughout the proof we employ the abbreviations $h_1=h_1(\vk)$ and $R_0=R_0(2\vk)$.

Looking at \eqref{E:h2 exp}, we are lead naturally to estimate
\begin{align}
\vk^6 \biggl| \iint \bigl[\partial^2 R_0 \psi^{12}\bigr] \bigl[(h_1')^2 + 2 (h_1^2)'' \bigr] \,dx\,dt \biggr|
	&\lesssim  \vk^4 \| \psi^6 h_1' \|_{L^2_{t,x}}^2 + \vk^4  \| \psi^6 h_1 \|_{L^2_{t,x}}^2  \notag\\
&\lesssim \vk^{-\frac23} \bigr\{ \delta^2 + \|q\|_{\vLS}^2 \bigr\}, \label{2for2a}
\end{align}
by using that $|\partial^2R_0 \psi^{12}|+ |\partial^4R_0 \psi^{12}| \lesssim \vk^{-2} \psi^{12}$ and Proposition~\ref{P:h1}.

Looking at the remaining terms in \eqref{E:h2 exp} and incorporating them into $I(\vk)$, we find ourselves needing to consider
\begin{align}\label{tilde I}
 5 \iint  \bigl\{ - \partial_x^2\bigl[ q^2 - 16 \vk^6 h_1^2\bigr] - 4 \vk^4 [h_1'']^2 + \bigl[(q')^2 - 16 \vk^6(h_1')^2\bigr] \bigr\} \psi^{12} \,dx\,dt.
\end{align}
To proceed, we shall estimate the three terms inside the braces, working from left to right.

The first term makes a negligible contribution:  As $q=-\vk(-\partial_x^2+4\vk^2) h_1$, so
$$
q^2 - 16 \vk^6 h_1^2 = \vk^2 (h_1'')^2 - 4\vk^4 (h_1^2)'' + 8\vk^4 (h_1')^2 .
$$
Thus by Proposition~\ref{P:h1}, we see that
\begin{align}
\biggl| \iint  \bigl\{  \partial_x^2\bigl[ q^2 - 16 \vk^6 h_1^2\bigr]  \bigr\} \psi^{12} \,dx\,dt\biggr|
	&\lesssim \vk^2 \| \psi h_1'' \|_{L^2_{t,x}}^2 + \vk^4  \| \psi h_1 \|_{L^2_t H^1 }^2 \notag\\
&\lesssim \vk^{-\frac23} \bigr\{ \delta^2 + \|q\|_{\vLS}^2 \bigr\}. \label{2for2b}
\end{align}

Turning now to the middle term from \eqref{tilde I}, we see from Lemma~\ref{L:2xC} that
\begin{align*}
\bigl\| \tfrac1{\psi^2} \bigl[R_0 \psi^{12} R_0 - \psi^6 R_0^2 \psi^6\bigr]\tfrac1{\psi^2} \bigr\|_{H^{-1}_\vk \to H^1_\vk} \lesssim \vk^{-4}.
\end{align*}
Thus, recalling \eqref{h1 nice} and Lemma~\ref{L:change},
\begin{align*}
\biggl\| \int 4 \vk^4 [h_1'']^2 \psi^{12} \,dx - 4 \vk^2 \bigl\langle \psi^6 q'', R_0^2 \psi^6 q''\bigr\rangle\biggr\|_{L^1_t}
	&\lesssim \vk^{-2}  \| \psi^2 q '' \|_{L^2_t H^{-1}_\vk}^2 \\
&\lesssim \vk^{-2} \bigr\{ \delta^2 + \|q\|_{\vLS}^2 \bigr\}.
\end{align*}
Continuing from here, we note that by Lemma~\ref{L:change},
\begin{align*}
4\vk^2 \Bigl\| &\bigl\langle \psi^6 q'', R_0^2 \psi^6 q''\bigr\rangle - \bigl\langle (\psi^6 q)'', R_0^2 (\psi^6 q)''\bigr\rangle\Bigr\|_{L^1_t} \\
&\lesssim \Bigl( \| \psi q'' \|_{L^2_t H^{-1}_\vk} + \| \psi q' \|_{L^2_t H^{-1}_\vk} + \| \psi q \|_{L^2_t H^{-1}_\vk} \Bigr)
			\Bigl(\| \psi q' \|_{L^2_t H^{-1}_\vk} + \| \psi q \|_{L^2_t H^{-1}_\vk} \Bigr) \\[2mm]
&\lesssim \vk^{-\frac13} \bigr\{ \delta^2 + \|q\|_{\vLS}^2 \bigr\}.
\end{align*}
Putting these two pieces together, we find
\begin{align}\label{2for2c}
\int \biggl| \int 4 \vk^4 [h_1'']^2 \psi^{12} \,dx - 4 \vk^2 \bigl\langle (\psi^6 q)'', R_0^2 (\psi^6 q)''\bigr\rangle\biggr|\,dt
	&\lesssim \vk^{-\frac13} \bigr\{ \delta^2 + \|q\|_{\vLS}^2 \bigr\}.
\end{align}

Proceeding in a similar way, we find that
\begin{align*}
\biggl\| \int 16 \vk^6 [h_1']^2 \psi^{12} \,dx - 16 \vk^4 \bigl\langle \psi^6 q', R_0^2 \psi^6 q'\bigr\rangle\biggr\|_{L^1_t}
	&\lesssim \| \psi^2 q ' \|_{L^2_t H^{-1}_\vk}^2 \\
&\lesssim \vk^{-\frac13} \bigr\{ \delta^2 + \|q\|_{\vLS}^2 \bigr\}
\end{align*}
and also that
\begin{align*}
&\Bigl\| 	\bigl\langle (\psi^6 q')', \bigl[\Id + 4\vk^2 R_0\bigr]R_0 (\psi^6 q')'\bigr\rangle
		- \bigl\langle (\psi^6 q)'', \bigl[\Id + 4\vk^2 R_0\bigr]R_0 (\psi^6 q)''\bigr\rangle\Bigr\|_{L^1_t}  \\
&\lesssim \Bigl( \| \psi q'' \|_{L^2_t H^{-1}_\vk} + \| \psi q' \|_{L^2_t H^{-1}_\vk} + \| \psi q \|_{L^2_t H^{-1}_\vk} \Bigr)
		\Bigl(  \| \psi q' \|_{L^2_t H^{-1}_\vk} + \| \psi q \|_{L^2_t H^{-1}_\vk} \Bigr) \\[1mm]
&\lesssim \vk^{-\frac13} \bigr\{ \delta^2 + \|q\|_{\vLS}^2 \bigr\}.
\end{align*}
Hence, writing $\Id - 16\vk^4R_0^2 = - \partial\bigl[\Id + 4\vk^2 R_0\bigr]R_0\partial$, we find
\begin{equation}\label{2for2d}
\begin{aligned}
\Bigl\| \int \bigl[ (q')^2 - 16 \vk^6 (h_1')^2\bigr] \psi^{12} \,dx
		- \bigl\langle (\psi^6 q)'', \bigl[\Id + 4&\vk^2 R_0\bigr] R_0 (\psi^6 q)''\bigr\rangle\Bigr\|_{L^1_t} \\
&\lesssim \vk^{-\frac13} \bigr\{ \delta^2 + \|q\|_{\vLS}^2 \bigr\}.
\end{aligned}
\end{equation}

Aggregating our estimates \eqref{2for2a}, \eqref{2for2b}, \eqref{2for2c}, and \eqref{2for2d}, we discover that
\begin{align}
\biggl| I - 5 \int \bigl\langle (\psi^6 q)'', R_0 (\psi^6 q)''\bigr\rangle\,dt \biggr| \lesssim \vk^{-\frac13} \bigr\{ \delta^2 + \|q\|_{\vLS}^2 \bigr\},
\end{align}
which proves \eqref{2for2}.  We now turn our attention to \eqref{2for2'}

Comparing the definition of $I(\vk)$ with $\int - j_1 \psi^{12}$, we see two discrepancies. The first is easily estimated via integration by parts and Proposition~\ref{P:h1}:
\begin{align}
\biggl| 8\vk^4  \iint \psi^{12} \bigl[ h_1 h_1^{(4)} - (h_1'')^2 \bigr]\,dx\,dt \biggr|
	&\lesssim \vk^4  \bigl\{ \| \psi h_1' \|_{L^2_{t,x}} + \| \psi h_1 \|_{L^2_{t,x}} \bigr\} \| \psi h_1'' \|_{L^2_{t,x}} \notag\\
&\lesssim \vk^{-\frac13} \bigr\{ \delta^2 + \|q\|_{\vLS}^2 \bigr\}  . \label{discrep1}
\end{align}
Regarding the second discrepancy, we write $\Id - 4\vk^2R_0(2\vk)=-\partial_x^2R_0(2\vk)$ and then integrate by parts to estimate it as follows:
\begin{align*} 
\biggl|  \iint \bigl[\partial_x^2R_0(2\vk) \psi^{12}\bigr] \bigl[5 (q')^2 - 5 (q^2)''\bigr] \,dx\,dt \biggr| &\lesssim \vk^{-2} \| \psi q \|_{L^2_t H^1}^2
	\lesssim \vk^{-\frac23} \bigr\{ \delta^2 + \|q\|_{\vLS}^2 \bigr\}.
\end{align*}
Together with \eqref{discrep1}, this estimate allows us to deduce \eqref{2for2'} from \eqref{2for2}.
\end{proof}

\begin{lemma}[Cubic current]\label{L:cubic}
For $q:[-1,1]\to B_\delta\cap\Schw(\R)$ and $\delta$ sufficiently small,
\begin{align}\label{E:P:cubic}
\biggl| \int_{-1}^1 \int_{-\infty}^\infty  j_2 \psi^{12} \,dx\,dt \biggr| \lesssim \vk^{-\frac16} \delta \bigr\{ \delta^2 + \|q\|_{\vLS}^2 \bigr\}.
\end{align}
\end{lemma}

\begin{proof}
Our first two reductions are based on Lemma~\ref{L:2xC} and the representation
\begin{align}\label{3red0}
\iint 64\vk^7 h_3 \psi^{12} \,dx\,dt= - 64\vk^7 \int \Tr\Bigl\{ R_0 \psi^{12} R_0 q R_0 q R_0 q \Bigr\} \,dt .
\end{align}
Here and below $R_0=R_0(\vk)$.

Employing \eqref{R0 vs psi sym}, we may write 
\begin{align*}
\Tr\Bigl\{ \bigl[R_0 \psi^{12} R_0 - \psi^6 R_0^2 \psi^6\bigr] q R_0 q R_0 q \Bigr\}
	= \vk^{-2} \sum \Tr\Bigl\{ A_i' \psi^3 q \bigl[\psi^3 R_0 \tfrac1{\psi^3} \bigr] \psi^3 q R_0 q\psi^6 A_i^{\ }  \Bigr\},
\end{align*}
where the sum has finitely many terms and each of the operators $A_i$ and $A_i'$ satisfy \eqref{A:2xC}.
This allows us to apply Proposition~\ref{P:X_l thingy} and so deduce that
\begin{align}
\biggl| \vk^7  \!\! \int \! \Tr\Bigl\{ \bigl[R_0 \psi^{12} R_0 - \psi^6 R_0^2 \psi^6\bigr] q R_0 q R_0 q \Bigr\} \,dt\biggr|
	\lesssim \vk^{-\frac{5}{6}} \delta \bigr\{ \delta^2 + \|q\|_{\vLS}^2 \bigr\}, \label{3red1}
\end{align}
which constitutes an acceptable error.

Next we seek to apply \eqref{psi vs q vs R0} in a similar way to prove
\begin{align}
\biggl| \vk^7  \!\! \int \! \Tr\Bigl\{ R_0 q\psi^4 \Bigl( \psi^2 R_0 q R_0 \psi^2 - R_0 \psi^4 q R_0 \Bigr) q\psi^4 R_0 \Bigr\} \,dt\biggr|
	\lesssim \vk^{-\frac12} \delta \bigr\{ \delta^2 + \|q\|_{\vLS}^2 \bigr\}.
\label{3red2}
\end{align}
The terms involving $A_i$ are readily seen to be acceptable via Proposition~\ref{P:X_l thingy}.  However this leaves us to prove
\begin{equation*}
\biggl| \vk^7  \! \int \Tr\Bigl\{R_0 q\psi^4 \Bigl( \vk^2 R_0^2 [\psi^3\psi' q]' R_0^2 + \partial R_0^2 [\psi^3\psi' q]' \partial R_0^2 \Bigr) q\psi^4 R_0  \Bigr\} \,dt\biggr|
\lesssim \text{RHS\eqref{3red2}}.
\end{equation*}
To do this, we cycle the trace and apply H\"older's inequality, using the following two inputs:  First, by \eqref{E:X_l thingy}, we have
\begin{align*}
\bigl\| R_0 q\psi^4  R_0^2 \psi^4 q \sqrt{R_0} \bigr\|_{L^2_t \mathfrak{I}_2}^2
	&= \bigl\| R_0 q\psi^4  R_0^2 \psi^4 q R_0 q\psi^4  R_0^2 \psi^4 q R_0 \bigr\|_{L^1_t \mathfrak{I}_1} \\
&\lesssim \vk^{-11-\frac{1}{3}} \delta^2 \bigr\{ \delta^2 + \|q\|_{\vLS}^2 \bigr\}.
\end{align*}
Second, using Lemma~\ref{L:HS} and \eqref{E:L:change'}, we see that
\begin{align*}
\bigl\| \vk^2 R_0^{3/2} [\phi q]' R_0 + \partial R_0^{3/2} [\phi q]' \partial R_0   \bigr\|_{L^2_t \mathfrak{I}_2}^2
	&\lesssim  \vk^{-3} \bigl\| [\phi q]' \bigr\|_{L^2_t H^{-1}_\vk}^2 \\
&\lesssim \vk^{-\frac{11}{3}} \bigr\{ \delta^2 + \|q\|_{\vLS}^2 \bigr\}
\end{align*}
for any Schwartz function $\phi$.

Writing $u:=\psi^4 q$ and combining \eqref{3red0}, \eqref{3red1}, and \eqref{3red2}, we finally achieve our sought-after reduction: To prove the lemma, it suffices to show that
\begin{align}\label{3red4}
\int\biggl| 64\vk^7\Tr\Bigl\{  R_0 u R_0 u R_0 u R_0\Bigr\} - \int 10 u^3\,dx\biggr| \,dt
 	\lesssim  \vk^{-\frac16} \delta \bigr\{ \delta^2 + \|q\|_{\vLS}^2 \bigr\}.
\end{align}

In order to exhibit the required cancellation, it is convenient to freeze the time variable and show instead that
\begin{align}\label{3red5}
\biggr| 64\vk^7 \Tr\Bigl\{  R_0 u R_0 u R_0 u R_0\Bigr\}  - \int 10 u^3\,dx \biggr|
		\lesssim  \vk^{-\frac16} \delta \Bigl\{ \delta^2 + \| u'' \|_{H^{-1}_\vk}^2 \Bigr\}.
\end{align}
This then yields \eqref{3red4} by integrating in time and applying Lemma~\ref{L:change}.  We begin by writing out the trace as a paraproduct with an explicit symbol.  Concretely, if
\begin{align*}
m(\eta_1,\eta_2) = \int_\R \frac{d\xi}{2\pi[\xi^2+\vk^2]^2[(\xi+\eta_1)^2+\vk^2][(\xi+\eta_2)^2+\vk^2]},
\end{align*}
then using that $u$ is real-valued, we may evaluate the trace (in Fourier variables) as follows:
$$
\Tr\bigl\{  R_0 u R_0 u R_0 u R_0\bigr\} = \tfrac{1}{\sqrt{2\pi}}\iint m(\eta_1,\eta_2) \overline{\hat u(\eta_1+\eta_2)} \hat u(\eta_2) \hat u(\eta_1)\,d\eta_1\,d\eta_2.
$$
By comparison, 
$$
\int u^3\,dx  = \tfrac{1}{\sqrt{2\pi}}\iint \overline{\hat u(\eta_1+\eta_2)} \hat u(\eta_2) \hat u(\eta_1)\,d\eta_1\,d\eta_2 .
$$

By contour integration (or partial fractions), we find
\begin{align*}
m= \tfrac{640 \vk^6+48\vk^4 [2 \eta_1^2 + 2 \eta_2^2 + (\eta_1-\eta_2)^2]+4\vk^2[3 \eta_1^4 - 4 \eta_1^3 \eta_2 + 4 \eta_1^2 \eta_2^2 - 4 \eta_1 \eta_2^3 + 3 \eta_2^4]+\eta_1^2 \eta_2^2 (\eta_1-\eta_2)^2 }{4\vk^3[\eta_1^2+4\vk^2]^2 [\eta_2^2+4\vk^2]^2 [(\eta_1-\eta_2)^2+4\vk^2]}.
\end{align*}
With patient computation (we recommend collecting terms by power of $\vk$), this yields the key cancellation:
\begin{align*}
\bigl| 64\vk^7 m - 10\bigr| \lesssim  \frac{\eta_1^2\eta_2^2}{[\eta_1^2+4\vk^2] [\eta_2^2+4\vk^2]} + \frac{\vk^4[\eta_1^2+\eta_2^2+(\eta_1-\eta_2)^2]}{[\eta_1^2+4\vk^2] [\eta_2^2+4\vk^2] [(\eta_1-\eta_2)^2+4\vk^2]} 
\end{align*}
uniformly for $\eta_1,\eta_2\in\R$.

To estimate the contribution of the first term, we use Cauchy--Schwarz:
\begin{align*}
& \iint \frac{\eta_1^2\eta_2^2 |\hat u(\eta_1+\eta_2) \hat u(\eta_2) \hat u(\eta_1)|}{[\eta_1^2+4\vk^2] [\eta_2^2+4\vk^2]} \,d\eta_1\,d\eta_2 \\
&\lesssim \| u'' \|_{H^{-1}_\vk}^2 \biggl(\iint \frac{|\hat u(\eta_1+\eta_2)|^2}{[\eta_1^2+4\vk^2] [\eta_2^2+4\vk^2]}\,d\eta_1\,d\eta_2\biggr)^{1/2} 
\lesssim \vk^{-1/2} \| u'' \|_{H^{-1}_\vk}^2 \| u \|_{H^{-1}_\vk} .
\end{align*}

Exploiting the $\eta_1\leftrightarrow\eta_2$ symmetry, the contribution of the second term is controlled by
\begin{align}\label{3red6}
\iint_{|\eta_2|\leq|\eta_1|} \frac{ \eta_1^2 |\hat u(\eta_1+\eta_2) \hat u(\eta_2) \hat u(\eta_1)|\,d\eta_1\,d\eta_2 }{[\eta_1^2+4\vk^2]}.
\end{align}
We split this integral into two parts depending on whether $|\eta_1+\eta_2| > |\eta_2|$ or conversely, $|\eta_1+\eta_2| < |\eta_2|$.  In the former case, our next inequality is elementary; in the latter, one must first make the change of variables $\zeta_1=\eta_1$, $\zeta_2=-\eta_1-\eta_2$ to obtain
\begin{align*}
\eqref{3red6}&\lesssim \iint \frac{ \eta_1^2 |\hat u(\eta_1)|}{\sqrt{\eta_1^2+4\vk^2}}
		\frac{ \sqrt{(\eta_1+\eta_2)^2 + 1}\,|\hat u(\eta_1+\eta_2)|}{\sqrt{(\eta_1+\eta_2)^2+4\vk^2}}
		\frac{ |\hat u(\eta_2)|}{\sqrt{\eta_2^2+1}} \,d\eta_1\,d\eta_2 .
\end{align*}
Although used to arrive at this final form of the integrand, we have now abandoned the constraints on $\eta_1,\eta_2$.  To complete our estimation of \eqref{3red6}, we now seek to employ Schur's test. Setting  $N=\vk^{1/3}\in[1,\vk]$, Cauchy--Schwarz shows
\begin{align*}
\sup_\zeta \int &\frac{ \sqrt{(\eta+\zeta)^2 + 1}\,|\hat u(\eta+\zeta)|}{\sqrt{(\eta+\zeta)^2+4\vk^2}} \,d\eta \\
&= \int_{|\xi|\leq N} \frac{ \xi^2 + 1}{\sqrt{\xi^2+4\vk^2}} \frac{ |\hat u(\xi)|}{\sqrt{\xi^2 + 1}} \,d\xi
			+ \int_{|\xi|\geq N} \frac{ \sqrt{\xi^2 + 1}}{\xi^2} \frac{ \xi^2\,|\hat u(\xi)|}{\sqrt{\xi^2+4\vk^2}} \,d\xi \\[2mm]
&\lesssim \vk^{-1/6} \Bigl\{ \| u \|_{H^{-1}} + \| u'' \|_{H^{-1}_\vk} \Bigr\}.
\end{align*}
Employing this bound in Schur's test, we deduce that \eqref{3red6}${}\lesssim{}$RHS\eqref{3red5}.  This completes the proof of \eqref{3red5} and with that, of the lemma.
\end{proof}

With these preliminaries complete, we are now ready to prove the local smoothing estimate.

\begin{proof}[Proof of Proposition~\ref{P:LS5}]
As noted at the beginning of this section, it suffices to prove the result for $\vk$ large (relative to some absolute constant).

Looking back to the basic identity \eqref{E:int micro} for the localized conservation law and employing Lemmas~\ref{LS:j3},~\ref{L:2for2}, and~\ref{L:cubic}, we find that
\begin{align}\label{303}
\| (\psi^6 q )''\|_{L^2_t H^{-1}_\vk([-1,1]\times\R) }^2 \lesssim \|\rho\|_{L^\infty_t L^1_x}  + \vk^{-\frac16} \bigr\{ \delta^2 + \|q\|_{\vLS}^2 \bigr\}.
\end{align}
On the other hand, $\rho\geq0$, so by \eqref{E:rho alpha} and Proposition~\ref{P:H-1 bound}, we deduce
\begin{align}\label{304}
\| (\psi^6 q )''\|_{L^2_t H^{-1}_\vk([-1,1]\times\R) }^2 \lesssim \|q(0)\|_{H^{-1}_\vk}^2  + \vk^{-\frac16} \bigr\{ \delta^2 + \|q\|_{\vLS}^2 \bigr\}.
\end{align}
The final estimate \eqref{E:P:LS5} now follows by taking a supremum over all (spatial) translates of the solution $q$ and choosing $\vk$ sufficiently large.
\end{proof}

\begin{corollary}\label{C:LS5}
Fix $\delta>0$ sufficiently small and $\phi\in\Schw(\R)$. For every initial data $q\in B_\delta\cap \Schw(\R)$, the corresponding solution $q(t)$ to \eqref{5th} satisfies
\begin{equation}\label{E:C:LS5}
\sup_{x_0,t_0\in\R}\ \int_{t_0-1}^{t_0+1}\int_{-\infty}^\infty  \phi(x-x_0)^2 \bigl[ |q'(t,x)|^2 + |q(t,x)|^2 \bigr]\,dx\,dt \lesssim_\phi \delta^2 .
\end{equation}
Moreover, given any $Q\subset B_\delta\cap \Schw(\R)$ that is $H^{-1}(\R)$-equicontinuous,
\begin{equation}\label{E:C:LS5'}
\lim_{R\to\infty} \ \sup_{q\in Q}\ \sup_{x_0,t_0\in\R}\ \int_{t_0-1}^{t_0+1}\int_{-\infty}^\infty  \tfrac1R \phi\bigl(\tfrac{x-x_0}R\bigr)^2 \bigl[ |q'(t,x)|^2 + |q(t,x)|^2 \bigr]\,dx\,dt  = 0.
\end{equation}
\end{corollary}

\begin{proof}
The supremum over $t_0,x_0$ is ultimately a red herring, because the space-time translation can be transferred to $q$.  The boundedness and equicontinuity of the correspondingly larger set of initial data was demonstrated in Proposition~\ref{P:H-1 bound}.

To prove \eqref{E:C:LS5}, we first observe that
\begin{align}\label{breakfast}
\iint \phi(x)^2 \bigl[ |q'(t,x)|^2 + |q(t,x)|^2 \bigr]\,dx\,dt &\lesssim \| \phi q \|_{L^2_t H^1}^2 + \| \phi' q \|_{L^2_t H^1}^2
\end{align}
and then that
\begin{align}\label{breakfast'}
\| \phi q \|_{L^2_t H^1}^2 \lesssim \| \phi q \|_{L^\infty_t H^{-1}}^2 + \| (\phi q)'' \|_{L^2_t H^{-1}}^2 
\end{align}
(as well as the analogous assertion with $\phi\mapsto\phi'$).  In this way, \eqref{E:C:LS5} follows from the $\vk=1$ cases of Lemma~\ref{L:change}, Proposition~\ref{P:LS5}, and \eqref{E:P:H-1}.

We turn now to \eqref{E:C:LS5'} and adopt the notation $\phi_R(x)=R^{-1/2} \phi(x/R)$.  As it is our intention to employ \eqref{breakfast}, we should also consider $\phi_{R}' = R^{-3/2} \phi'(x/R)$ in what follows; however, given the generality afforded $\phi$, this is covered by the same analysis.

By first looking on the Fourier side, and then applying \eqref{elem mult}, we find
\begin{align}\label{lunch}
\| \phi_R q \|_{H^1}^2 \lesssim \vk^4  \| \phi_R q \|_{H^{-1}}^2 + \| (\phi_R q)'' \|_{H^{-1}_\vk}^2 \lesssim \tfrac{\vk^4}{R}  \| q \|_{H^{-1}}^2 + \| (\phi_R q)'' \|_{H^{-1}_\vk}^2
\end{align}
uniformly for $\vk\geq 1$.

We now focus our attention on the right-most term in \eqref{lunch}. As
\begin{align}
\| f \|_{H^{-1}_\vk}^2 = \tfrac{512}{7} \int_\R\bigl\langle \psi_z^{6} f, \bigl[\tfrac{1}{\psi_z^6}R_0(2\vk)\psi_z^{6}\bigr] \psi_z^{6} f\bigr]\rangle\,dz
\end{align}
for any $f$, so it follows from Lemma~\ref{L:w commutator} that
\begin{align}
\| (\phi_R q)'' \|_{H^{-1}_\vk}^2 \lesssim  \int_\R \ \bigl\| \phi_R \psi_z^6 q'' \bigr\|_{H^{-1}_\vk}^2 + \bigl\|\phi_R' \psi_z^6 q' \bigr\|_{H^{-1}_\vk}^2
		+ \bigl\| \phi_R'' \psi_z^6 q \bigr\|_{H^{-1}_\vk}^2\,dz.
\end{align}
Next we apply \eqref{elem mult}.  Noting that
\begin{align}
\int_\R \| \phi_R \psi_z^3 \|_{H^2}^2 + \| \phi_R' \psi_z^3 \|_{H^2}^2 + \| \phi_R'' \psi_z^3 \|_{H^2}^2 \,dz \lesssim 1
\end{align}
uniformly in $R$, we are lead to the conclusion
\begin{align}
\| (\phi_R q)'' \|_{H^{-1}_\vk}^2 \lesssim  \sup_z\Bigl\{ \ \bigl\| (\psi_z^3 q)'' \bigr\|_{H^{-1}_\vk}^2 + \bigl\| \psi_z^3 q' \bigr\|_{H^{-1}_\vk}^2
		+ \bigl\| \psi_z^3 q \bigr\|_{H^{-1}_\vk}^2 \Bigr\},
\end{align}
uniformly in $R$.  By applying  \eqref{E:L:change} and \eqref{E:L:change'}, we then obtain
\begin{align}
\| (\phi_R q)'' \|_{L^2_t H^{-1}_\vk}^2 \lesssim  \|q \|_{L^\infty_t H^{-1}_\vk}^2 + \vk^{-\frac23} \|q \|_{L^\infty_t H^{-1}}^2  + \|q\|_{\vLS}^2. 
\end{align}

Returning to \eqref{lunch} and employing Propositions~\ref{P:H-1 bound} and~\ref{P:LS5} we now deduce that
\begin{align}\label{lunch'}
\| \phi_R q \|_{L^2_t H^1}^2 \lesssim \tfrac{\vk^4}{R}  \| q(0) \|_{H^{-1}}^2 + \|q(0)\|_{H^{-1}_\vk}^2 + \vk^{-\frac16} \delta^2.
\end{align}
This quantity can be made arbitrarily small, uniformly in $q$, by choosing $\vk$ large and then $R$ even larger still.  The uniformity here uses the equicontinuity of the set of initial data.  
\end{proof}

\section{Compactness}\label{S:C}

This section is devoted to proving a key compactness property of solutions to \eqref{5th}:

\begin{proposition}\label{P:compact} Fix $\delta>0$ sufficiently small and let $Q\subseteq B_\delta \cap \Schw(\R)$ be precompact in $H^{-1}(\R)$.  Then
$$
Q_* = \bigl\{ e^{tJ\nabla H_\text{5th}} q : q\in Q\text{ and } t\in[-1,1]\bigr\}
$$
is also precompact in $H^{-1}(\R)$.
\end{proposition}

Evidently, this conclusion would follow from Theorem~\ref{T:main}, because the continuous image of the compact set $[-1,1]\times Q$ is compact.  However, we will need this compactness result in order to prove that theorem.  Its principal role is to lessen the continuity requirements we need to show on sequences of solutions, by reducing the question of norm convergence to one of weak convergence.

As discussed earlier, precompactness comprises three ingredients: boundedness, equicontinuity, and tightness.  The first two follow from Proposition~\ref{P:H-1 bound}; our central enemy in this section is tightness. 

In order to control the transport of the $H^{-1}$ norm of a solution, it is convenient to employ the conserved density $\rho$.  While it has been shown previously that $\int \rho$ controls the global $H^{-1}$ norm (cf. \eqref{E:rho alpha}), we need such an equivalence that holds \emph{locally} in space.  This is new and the subject of our next lemma:
 
\begin{lemma}\label{L:q vs rho}
Fix $\delta>0$ sufficiently small and $w:\R\to(0,\infty)$ that satisfies
\begin{equation}\label{E:w hyp'}
|w''(x)| + |w'(x)|\leq w(x) \qtq{and} \frac{w(y)}{w(x)}\leq e^{|x-y|/2} .
\end{equation}
Then for $\vk_0\geq 1$ sufficiently large {\upshape(}independent of $w${\upshape)},
\begin{equation}\label{E:rho_local}
\tfrac12 \| w q \|_{H_\varkappa^{-1}(\R)}^2 \leq    \| w^2 \rho(\vk) \|_{L^1(\R)}
	\leq 2 \| w q \|_{H_\varkappa^{-1}(\R)}^2
\end{equation}
uniformly for $\vk\geq \vk_0$ and $q\in B_\delta\cap\Schw(\R)$.
\end{lemma}

\begin{proof}
Notice that \eqref{E:w hyp'} guarantees that \eqref{E:w commutator} holds with an absolute constant.  This will be important for ensuring that $\vk_0$ does not depend on $w$.

Recall that for $\vk\geq 1$ and $\delta$ sufficiently small, both $\rho\geq 0$ and $g\geq 0$.
From this and \eqref{E:h in Linfty}, we may then deduce that for  $\delta$ sufficiently small,
\begin{align*}
\bigl\| [\rho-2\vk g\rho] w^2 \bigr\|_{L^1} \lesssim \vk^{-\frac12} \bigl\| \rho \,w^2 \bigr\|_{L^1}
	\lesssim \vk^{-\frac12} \bigl\| 2\vk g\rho\, w^2\bigr\|_{L^1} .
\end{align*}
In this way, we see that it suffices to prove \eqref{E:rho_local} with $\rho$ replaced by $2\vk g\rho$.

The benefit of this reduction (indeed equivalency) is that it removes $g$ from the denominator:
\begin{align}\label{2vkgrho}
2\vk g\rho = \bigl( 4\vk^3 h_2 - 8\vk^4 h_1^2\bigr) - 8\vk^4 \bigl(g-\tfrac1{2\vk}-h_1\bigr)h_1 + 4\vk^3 \sum_{\ell\geq 3} h_\ell .
\end{align}

From Lemmas~\ref{L:w commutator} and~\ref{L:HS} we see that
\begin{align}\label{grho1}
\int h_2 w^2 \,dx = \Tr\Bigl\{\sqrt{R_0(\vk)} w q R_0(\vk)^2 w q \sqrt{R_0(\vk)}\Bigr\} + O\Bigl( \vk^{-4} \| w q \|_{H_\varkappa^{-1}}^2 \Bigr)
\end{align}
and analogously that
\begin{align}\label{grho2}
\sum_{\ell\geq 3} \biggl| \int h_\ell w^2 \,dx \biggr| \lesssim \vk^{-\frac{9}2} \delta \| w q \|_{H_\varkappa^{-1}}^2 .
\end{align}
Next we employ Lemma~\ref{L:w commutator} and \eqref{h1 nice} in a similar fashion to see that
\begin{align}\label{grho3}
\int h_1^2 w^2 \,dx = \vk^{-2} \bigl\langle w q,\ R_0(2\vk)^2 w q\bigr\rangle + O\Bigl( \vk^{-5} \| w q \|_{H_\varkappa^{-1}}^2 \Bigr).
\end{align}

As our last preliminary before treating the terms in \eqref{2vkgrho}, we observe that the techniques just used show
\begin{align}\label{grho4}
\bigl\| [g-\tfrac1{2\vk}-h_1] w^2 \bigr\|_{L^1} \lesssim \vk^{-3} \| w q \|_{H_\varkappa^{-1}}^2 .
\end{align}
We use this with \eqref{E:h in Linfty} to handle the middle term in \eqref{2vkgrho}.

Putting everything together, we find that
\begin{align*}
 \int 2\vk g\rho w^2 \,dx &= 4\vk^3 \Tr\Bigl\{\sqrt{R_0(\vk)} w q R_0(\vk)^2 w q \sqrt{R_0(\vk)}\Bigr\} \\
&\quad -8\vk^2\bigl\langle w q,\ R_0(2\vk)^2 w q\bigr\rangle + O\Bigl( \vk^{-\frac12} \| w q \|_{H_\varkappa^{-1}}^2 \Bigr) \\
&= \| w q \|_{H_\varkappa^{-1}}^2  + O\Bigl( \vk^{-\frac12} \| w q \|_{H_\varkappa^{-1}}^2 \Bigr).
\end{align*}
Note that the last step here just involves exact computation of the leading term, as can be done, for example, by differentiating the identity \eqref{E:L:HS} with respect to $\vk$.

The sought-after equivalence now follows by choosing $\vk$ sufficiently large relative to the (absolute) constant implicit in the big Oh notation.
\end{proof}

In what follows, we will use the following localization to large positive $x$:
\begin{align}
\Psi_{\!R}(x) := \tfrac12 + \tfrac12 \tanh\bigl(\tfrac{x-x_0(R)}{R}\bigr) \qtq{with} x_0(R)=R^2
\end{align}
and the corresponding localization $\Psi_{\!R}(-x)$ to large negative $x$.  The exact choice of $x_0(R)$ is not important; we merely require that $x_0/R\to\infty$ as $R\to\infty$.  Evidently, we have
\begin{align}\label{if we set}
\Psi_{\!R}'(x) := \tfrac{1}{2R} \phi_R(x)^2 \qtq{if we set} \phi_R(x) = \sech\bigl(\tfrac{x-x_0}R\bigr).
\end{align}
We write the derivative in this way, to draw an analogy with Corollary~\ref{C:LS5}.  We also note that the hypotheses of Lemma~\ref{L:q vs rho} are satisfied with $w(x)=\sqrt{\Psi_R(\pm x)}$, provided $R$ is sufficiently large.

Our last observation about this choice of cut-off is that
\begin{align}\label{1/Psi_R}
\bigl\| f/ \sqrt{\Psi_R}\, \bigr\|_{H^1_\vk} \lesssim \|f\|_{H^1_\vk}
\end{align}
uniformly for $\vk\geq 1$ and $f\in H^1(\R)$ with $\supp(f)\subseteq [R,\infty)$.

\begin{proof}[Proof of Proposition~\ref{P:compact}]
As noted earlier, the boundedness and equicontinuity of $Q_*$ follow from Proposition~\ref{P:H-1 bound}.  As well as their direct contribution to compactness, these properties will also play a key role in the proof of tightness.

From the compactness of $Q$ we find
\begin{align*}
\lim_{R\to\infty}\ \sup_{q\in Q} \ \bigl\| q \sqrt{\Psi_R}\, \bigr\|_{H^{-1}} =0.
\end{align*}
Indeed, as $Q$ is totally bounded, it suffices to verify this for individual $q$.  Combining this with Lemma~\ref{L:q vs rho}, we deduce that for $\vk$ sufficiently large, 
\begin{align}\label{compact in}
\lim_{R\to\infty}\ \sup_{q\in Q} \ \int_\R \rho(x;\vk,q) \bigl[\Psi_R(x)+\Psi_R(-x)\bigr]\,dx =0.
\end{align}
Henceforth, $\vk$ will remain fixed and implicit constants will be permitted to depend on it. In the converse direction, Lemma~\ref{L:Riesz} and \eqref{1/Psi_R} show us that the compactness of $Q_*$ will follow if we can prove
\begin{align}\label{compact out}
\lim_{R\to\infty}\ \sup_{q\in Q} \ \sup_{t\in[-1,1]} \int_\R \rho(x;\vk,q(t)) \bigl[\Psi_R(x)+\Psi_R(-x)\bigr] \,dx =0,
\end{align}
where $q(t)$ is defined from its initial data $q(0)=q\in Q$ by the $H_\text{5th}$-flow.

Comparing \eqref{compact in} and \eqref{compact out} and invoking the basic microscopic conservation law, we see that the proposition can be proved by showing
\begin{align}\label{compact goal}
\lim_{R\to\infty}\ \sup_{q\in Q} \ \biggl\| \int_\R j_\text{5th}(x;\vk,q(t)) \bigl[\Psi_R'(x)-\Psi_R'(-x)\bigr] \,dx \biggr\|_{L^1_t}=0.
\end{align}
This is what we shall do.  To improve readability, we will drop the $\Psi_R'(-x)$ term in what follows.  Its contributions may be handled in a parallel manner.

The analysis of \eqref{compact goal} will be much simpler than the parallel analysis in Section~\ref{S:LS}, because we no longer need to demonstrate decay in $\vk$.  This decay was essential for proving Corollary~\ref{C:LS5}, which we will now use to verify \eqref{compact goal}.  Recall that
\begin{align*}
j_\text{5th} &= \tfrac{2\vk}{g(\vk)} \bigl\{ 3 q^2 - 4 \vk^2 q - q'' - 16\vk^5 \bigl[g(\vk)-\tfrac{1}{2\vk}\bigr]\bigr\} \\
&\quad -  4\vk^2 R_0(2\varkappa) \big[q^{(4)} - 5 (q^2)'' + 5 (q')^2 + 10 q^3\big].
\end{align*}

Working our way through the terms in the first row using \eqref{if we set}, we have
\begin{gather*}
\tfrac1R \int\bigg|\int  \tfrac{6\vk}{g(\vk)} q^2 \phi_R^2 \,dx\bigg|\,dt
	\lesssim \bigl\| \tfrac1g \bigr\|_{L^\infty_{t,x}} \tfrac1R \| \phi_R q \|_{L^2_{t,x}}^2 \\
\tfrac1R \int\bigg|\int  \tfrac{8\vk^3}{g(\vk)} q \phi_R^2 \,dx\bigg|\,dt
	\lesssim \bigl\| \tfrac1g \bigr\|_{L^\infty_{t,x}} \| \phi_R \|_{L^2_{x}} \tfrac1R \| \phi_R q \|_{L^2_{t,x}} \\
\tfrac1R \int\bigg|\int  \tfrac{2\vk}{g(\vk)} q'' \phi_R^2 \,dx\bigg|\,dt
	\lesssim \bigl\| \tfrac1g \phi_R \bigr\|_{L^\infty_t H^1} \tfrac1R \Bigl\{ \| \phi_R q' \|_{L^2_{t,x}} + \| \phi_R' q \|_{L^2_{t,x}} \Bigr\} \\
\tfrac1R \int\bigg|\int  \tfrac{32\vk^6}{g(\vk)} [g(\vk)-\tfrac{1}{2\vk}\bigr] \phi_R^2 \,dx\bigg|\,dt
	\lesssim \tfrac1R  \bigl\| \tfrac1g \bigr\|_{L^\infty_{t,x}} \bigl\| g -\tfrac1{2\vk}\bigr\|_{L^2_{t,x}} \| \phi_R^2 \|_{L^2_{x}},
\end{gather*}
all of which are acceptable thanks to Corollary~\ref{C:LS5} and elementary calculations.

Looking now at the second row of terms in the definition of $j_\text{5th}$, we find that the operator $4\vk^2 R_0(2\vk)$ causes some irritation.  To handle this, we write
$$
4\vk^2 R_0(2\varkappa) \Psi_{\!R}' = \Psi_{\!R}' - R_0(2\varkappa)\Psi_{\!R}''' .
$$
While the operator remains in the second term, the two additional powers of $R^{-1}$ arising from the derivatives make this term trivial to handle.  Focusing instead on the dominant terms we estimate as follows:
\begin{gather*}
\tfrac1R \int\bigg|\int  q^{(4)} \phi_R^2 \,dx\bigg|\,dt
	\lesssim \| q\|_{L^\infty_t H^{-1}} \tfrac1R \| \phi_R^2 \|_{L^\infty_t H^5} \\
\tfrac1R \int\bigg|\int  5 (q^2)'' \phi_R^2 \,dx\bigg|\,dt 
	\lesssim \tfrac1R  \| \phi_R q' \|_{L^2_{t,x}} \| \phi_R' q \|_{L^2_{t,x}}\\
\tfrac1R \int\bigg|\int  5 (q')^2 \phi_R^2 \,dx\bigg|\,dt
	\lesssim \tfrac1R \| \phi_R q' \|_{L^2_{t,x}}^2 \\
\tfrac1R \int\bigg|\int  10 q^3 \phi_R^2 \,dx\bigg|\,dt
	\lesssim \tfrac1R  \| q\|_{L^\infty_t H^{-1}} \| \phi_R q \|_{L^2_t H^1}^2 .
\end{gather*}
Once again, these are readily seen to be acceptable via Corollary~\ref{C:LS5}.
\end{proof}

\section{Local smoothing for the difference flow}\label{S:dLS}

Our primary goal in this section is to prove the following local smoothing estimate for the difference flow, that is, the flow generated by the Hamiltonian $H_\text{5th} - H_\kappa$.  At the end of the section, we apply this to control how the two flows diverge from one another; see Corollary~\ref{C:dLS}.

\begin{proposition}[Local smoothing for the difference flow]\label{P:dLS}
For $\delta$ sufficiently small and $\kappa_0 \geq 1$ sufficiently large,
\begin{equation}\label{eq:local_smoothing_difference}
\bigl\| e^{tJ\nabla (H_\text{5th}-H_\kappa)} q \bigr\|^2_{\LS} \lesssim  1
\end{equation}
uniformly for $q\in B_\delta\cap \Schw(\R)$ and $\kappa\geq \kappa_0$.
\end{proposition}

Once again, we employ a spatially localized version of the conservation laws discussed in subsection~\ref{SS:diagonal}.  In this section, the parameter $\vk$ will be regarded as fixed (it would suffice to set $\vk=1$) and correspondingly, all implicit constants will be permitted to depend on it.  While we will be reusing many of the same estimates exhibited in Section~\ref{S:LS}, the nature of the cancellations involved is rather different.  Let us explain this more fully.

Recall that the currents \eqref{E:rho dot 5} and \eqref{E:rho dot k} split naturally into two parts, corresponding to the second and third summands in \eqref{E:rho defn}: in each formula, the top line originates in the time derivative of the diagonal Green's function, while the second line comes from $\partial_t q$.  In Section~\ref{S:LS}, the essential cancellations were between the two parts of $j_\text{5th}$.    Here, cancellations will arise between corresponding terms in $j_\text{5th}$ and $j_\kappa$.  In particular, the two parts of the currents may be treated independently and that is what we shall do.  The dominant part of the current $j_\text{5th}-j_\kappa$ comes from $\partial_t q$ and takes the form
\begin{align}\label{j0 defn}
j_0 :=& -  4\varkappa^2 R_0(2\varkappa) \Bigl[64\kappa^7\bigl[g(\kappa)-\tfrac1{2\kappa}\bigr] + 16\kappa^4q + 4\kappa^2 q'' + q^{(4)} \notag\\
&\qquad\qquad\qquad\qquad\qquad -12\kappa^2 q^2 - 5 (q^2)'' + 5 (q')^2 + 10 q^3\Bigr].
\end{align}
Regarding the remainder, we have that under the difference flow,
\begin{align}\label{jj defn}
\tfrac{d\ }{dt} \tfrac{\vk}{g(\vk)} = [j_\text{\rm 5th}-j_\kappa-j_0]' .
\end{align}

Note that the smoothing effect of $R_0(2\varkappa)$ in \eqref{j0 defn} is not helpful: this current is to be integrated against a bump function (the derivative of the localizing cutoff) and applying this operator to that bump function simply produces another bump function.  In fact, in order to simplify the treatment of this (the most significant) term, we shall adapt our localizing function $\Psi(x)$ accordingly:
\begin{align}\label{E:Psi6}
\Psi(x) := \frac{1}{4\vk^2} \int_x^\infty \bigl[(-\partial^2+4\vk^2)\psi^{12}\bigr](x')\,dx' .
\end{align}
Clearly, $\Psi'(x)$ is a Schwartz function.

As evidence that $j_0$ really does capture the dominant terms, we now show that the remainder can be estimated in a satisfactory way:

\begin{lemma}\label{L:dLS 1}
Fix $\vk\geq 1$ and $\phi\in\Schw(\R)$. If $\delta$ is sufficiently small, then
\begin{equation}\label{E:dLS 1}
\begin{aligned}
\biggl| \int_{-1}^1 \int_{-\infty}^\infty &[j_\text{\rm 5th}-j_\kappa-j_0](t,x) \phi(x) \,dx\,dt \biggr| \\
	&\lesssim \kappa^{-\frac12} \bigl\{ 1 + \| q \|_{\LS}^2\bigr\}
			+ \bigl\|\partial_x^2 \tfrac{1}{g(\vk)}\bigr\|_{L^\infty_t H^{-1}_\kappa}  \bigl\{ \delta + \| q \|_{\LS}\bigr\},
\end{aligned}
\end{equation}
uniformly for $q:[-1,1]\to B_\delta\cap\Schw(\R)$ and $\kappa\geq 1$.
\end{lemma}

\begin{proof}
After considerable rearrangement, \eqref{E:rho dot 5} and \eqref{E:rho dot k} yield
\begin{align}
j_\text{5th} - j_\kappa -j_0
	&= - \tfrac{2\vk}{g(\vk)} \Bigl\{ 16\kappa^5 \bigl[g(\kappa)-\tfrac1{2\kappa} \bigr]   + 4 \kappa^2 q + q'' - 3q^2 \Bigr\}  \label{diff current1} \\
&\quad - \tfrac{2\vk^3}{g(\vk)} \Bigl\{ 16\kappa^3 \bigl[g(\kappa)-\tfrac1{2\kappa} \bigr]   + 4 q  \Bigr\} \label{diff current2} \\
&\quad - \tfrac{2\vk^5}{g(\vk)} \Bigl\{ 16\kappa^{\ } \bigl[g(\kappa)-\tfrac1{2\kappa} \bigr]  \Bigr\}  \label{diff current3} \\
&\quad - \tfrac{2\vk^7}{g(\vk)} \Bigl\{ \tfrac{16}{\kappa^2-\vk^2} \bigl[\kappa g(\kappa)-\vk g(\vk) \bigr] \Bigr\}. \label{diff current4}
\end{align}
While this can be written in a more compact way, this expression helps highlight an underlying pattern.  In estimating the contribution of these terms, it will be convenient to employ an auxiliary $\tilde\varphi\in\Schw(\R)$, chosen so that $|\phi| \leq \tilde\varphi^2$. 

From \eqref{E:h in Linfty} and the diffeomorphism property (Lemma~\ref{L:diffeo}), we see that
$$
\kappa^{\frac32} \bigl\| g(\kappa) - \tfrac1{2\kappa} \bigr\|_{L^\infty_{t,x}} + \| g(\vk) \|_{L^\infty_{t,x}} +  \| \tfrac{1}{g(\vk)} - 2\vk \|_{L^\infty_t H^1} \lesssim 1
$$
and so the contributions of \eqref{diff current3} and \eqref{diff current4} are clearly acceptable.

We now turn our attention to \eqref{diff current1} and \eqref{diff current2}, which include cancellations.  The parts of these terms that do not involve cancellations are easily settled by using Corollaries~\ref{C:h in L^1} and~\ref{C:h in H^1}: 
$$
\sum_{\ell\geq 3} \kappa^5 \| \phi h_\ell(\kappa) \|_{L^1_{t,x}}  + \kappa^3 \| \tilde\varphi h_2(\kappa) \|_{L^2_{t,x}}
	\lesssim \kappa^{-\frac56} \Bigl\{ \delta^2 + \| q\|_{\LS}^2 \Bigr\}  .  
$$
This is acceptable since $\tfrac1{g(\vk)} \in L^\infty_{t,x}$ and $\tilde\varphi\in L^2_x$.

For the remaining part of \eqref{diff current2}, we use \eqref{h1 nice} and then \eqref{E:h1 with s}:
$$
 \| \phi [4\kappa^3 h_1(\kappa) + q] \|_{L^1_{t,x}} \lesssim  \kappa \| \tilde\varphi h_1'' \|_{L^2_{t,x}} \lesssim \kappa^{-1} \bigl\{ 1 + \| q\|_{\LS}^2 \bigr\}.
$$

This leaves us to exhibit two cancellations in \eqref{diff current1}. By \eqref{E:h2 exp} and the results of Proposition~\ref{P:h1},
\begin{align*}
\bigl\| \phi[ 16\kappa^5 h_2(\kappa) - 3q^2] \bigr\|_{L^1_{t,x}}
	&\lesssim \kappa^2 \| \tilde\varphi h_1 '' \|_{L^2_{t,x}}^2 + \kappa^4 \| \tilde\varphi h_1 '' \|_{L^2_{t,x}} \| \tilde\varphi h_1 \|_{L^2_{t,x}}
			+ \kappa^4  \| \tilde\varphi h_1 ' \|_{L^2_{t,x}}^2  \\
&\lesssim  \kappa^{-\frac23}\bigl\{ \delta^2 +  \| q \|_{\LS}^2 \bigr\}.
\end{align*}

It remains to handle the contribution of $16 \kappa^5 h_1(\kappa) + 4\kappa^2 q + q''$.  Using \eqref{E:h1 exp}, \eqref{h1 nice}, and integration by parts, its contribution simplifies to
\begin{equation}
\begin{aligned}
- \iint \tfrac{2\vk}{g(\vk)} \kappa h_1^{(4)}(\kappa) \phi\,dx\,dt &= 2\vk \int \bigl\langle \bigl(\tfrac{1}{g(\vk)}\bigr)'' , \phi R_0(2\kappa) q''\bigr\rangle \,dt \\
&\quad + 2 \vk \!\iint \bigl[ \tfrac{1}{g(\vk)} \phi'' + 2 \bigl(\tfrac{1}{g(\vk)}\bigr)' \phi' \bigr] \kappa h_1''(\kappa)\,dx \,dt.
\end{aligned}
\end{equation}
In the first line, we write $\phi R_0= R_0\phi + [\phi, R_0]$ and apply Lemma~\ref{L:basic_commutator} for the commutator.  Both the commutator term and the terms in the second line are easily seen to be
$$
O\Bigl( \kappa^{-1}\bigl\{ 1 +  \| q \|_{\LS}^2 \bigr\} \Bigr) 
$$
by using \eqref{E:h1 with s} and Cauchy--Schwarz.  On the other hand, \eqref{E:L:change'} shows that
$$
\biggl|\int \bigl\langle \bigl(\tfrac{1}{g(\vk)}\bigr)'' ,  R_0(2\kappa) \phi q''\bigr\rangle \,dt \biggr|
	\lesssim \bigl\|\partial_x^2 \tfrac{1}{g(\vk)}\bigr\|_{L^\infty_t H^{-1}_\kappa}  \bigl\{ \delta + \| q \|_{\LS}\bigr\} ,
$$
which is precisely the origin of the final term in \eqref{E:dLS 1}.
\end{proof}

Next, we demonstrate the key coercivity that we require.  The argument will be rather short, because we have deliberately styled our presentation in Section~\ref{S:LS} to make this possible.

\begin{lemma}\label{L:dLS 2}
Fix $\vk\geq 1$ and $\delta$ sufficiently small.  Then
\begin{gather*}
\| (\psi^6 q )''\|_{L^2_t H^{-1}_\kappa([-1,1]\times\R)}^2 
\lesssim  \int_{-1}^1 \int_{-\infty}^\infty  j_0(t,x) \Psi'(x) \,dx\,dt + \kappa^{-\frac16} \bigr\{ \delta^2 + \|q\|_{\LS}^2 \bigr\}
\end{gather*}
uniformly for $q:[-1,1]\to B_\delta\cap\Schw(\R)$ and $\kappa\geq 1$.
\end{lemma}

\begin{proof}
Recall that $\Psi$ was chosen so that $- 4\vk^2R_0(2\vk)\Psi'= \psi^{12}$.  Employing this and the series \eqref{E:g defn}, we find
\begin{align}
\iint j_0 \Psi' \,dx\,dt &= \iint \bigl[64\kappa^7h_1(\kappa) + 16\kappa^4q + 4\kappa^2 q'' + q^{(4)} \bigr] \psi^{12} \,dx\,dt \label{10:14}\\
&\qquad+ \iint \bigl[64\kappa^7h_2(\kappa) -12\kappa^2 q^2 - 5 (q^2)'' + 5 (q')^2 \big] \psi^{12} \,dx\,dt \label{10:15}\\
&\qquad+ \iint \biggl[10 q^3 + 64\kappa^7\sum_{\ell\geq 3} h_\ell(\kappa) \biggr] \psi^{12} \,dx\,dt . \label{10:16}
\end{align}

Applying \eqref{E:h1 exp} in \eqref{10:14}, we are lead to estimate
\begin{align*}
\biggl|\iint \kappa h_1^{(6)}(\kappa) \psi^{12} \,dx\,dt \biggr| \lesssim \kappa \| h_1 \|_{L^\infty_{t,x}} \lesssim \kappa^{-\frac12} \delta, 
\end{align*}
by integration by parts and \eqref{E:h in Linfty}.

To estimate \eqref{10:15}, we look to Lemma~\ref{L:2for2}.  Indeed, what appears here is greater  than $I(\kappa)$ due to the absence of the $-8\kappa^4[h_1''(\kappa)]^2$ term.

Looking now at \eqref{10:16}, we note that the $\ell=3$ term is handled by Lemma~\ref{L:cubic}, while $\ell\geq 4$ is bounded acceptably by Corollary~\ref{C:h in L^1}.
\end{proof}

\begin{proof}[Proof of Proposition~\ref{P:dLS}]  Choosing $\Psi$ as in \eqref{E:Psi6},  we begin with the identity
\begin{equation}\label{E:int micro'}
\int [\rho(1,x) - \rho(-1,x)] \Psi(x)\,dx = \int_{-1}^1 \int_{-\infty}^\infty [j_\text{5th}-j_\kappa](t,x) \Psi'(x)\,dx\,dt,
\end{equation}
valid for any Schwartz solution to the difference flow.  We then apply Lemmas~\ref{L:dLS 1} and~\ref{L:dLS 2} on the right-hand side; on the left-hand side, we use \eqref{E:rho alpha} and Proposition~\ref{P:H-1 bound}.  In this way, we find that
\begin{gather}\label{pool}
\| (\psi^6 q )''\|_{L^2_t H^{-1}_\kappa}^2 
\lesssim  1 + \kappa^{-\frac16} \|q\|_{\LS}^2 + \bigl\|\partial_x^2 \tfrac{1}{g(\vk)}\bigr\|_{L^\infty_t H^{-1}_\kappa} \bigl\{ \delta + \| q \|_{\LS}\bigr\}.
\end{gather}

Next we use Lemma~\ref{L:diffeo} and Proposition~\ref{P:H-1 bound} to see that
$$
\bigl\|\partial_x^2 \tfrac{1}{g(\vk)}\bigr\|_{L^\infty_t H^{-1}_\kappa} \lesssim \bigl\| \tfrac{1}{g(\vk)} - 2\vk \bigr\|_{L^\infty_t H^{1}_\vk}
	\lesssim \delta \lesssim 1
$$
and thence that 
\begin{equation*}
\bigl\|\partial_x^2 \tfrac{1}{g(\vk)}\bigr\|_{L^\infty_t H^{-1}_\kappa}  \| q \|_{\LS} \lesssim \eps^{-1} + \eps \| q \|_{\LS}^2 ,
\end{equation*}
uniformly for $\eps>0$.  The result now follows by plugging this into \eqref{pool}, taking a supremum over translates of $q$ (to recover the $\LS$ norm on the left-hand side) and finally by choosing $\kappa_0$ large enough and $\eps$ small enough. 
\end{proof}

Having proved our local smoothing estimate for the difference flow, we now demonstrate its role in the proof of well-posedness, namely, to show that the $H_\kappa$ flows closely track the full $H_\text{5th}$ flow (for $\kappa$ large).   This proximity is expressed through the reciprocal Green's function and in the weak topology.  These limitations will be removed in the next section by using the compactness demonstrated in Section~\ref{S:C}.

\begin{corollary}\label{C:dLS}
Fix $\vk\geq 1$ and $\delta>0$ sufficiently small. Given any $Q\subset B_\delta\cap \Schw(\R)$ that is $H^{-1}(\R)$-equicontinuous and any $\phi\in\Schw(\R)$, 
\begin{equation}\label{E:C:dLS}
\lim_{\kappa\to\infty} \ \sup_{q\in Q}\  \sup_{|t|\leq 1} \ \Bigl| \bigl\langle \phi,\bigl[ \tfrac{1}{g(t)} - \tfrac{1}{g_\kappa(t)} \bigr] \bigr\rangle \Bigr| = 0.
\end{equation}
Here we use the notations
\begin{align}\label{g flowed}
g(t) = g\bigl(x;\vk, e^{tJ\nabla H_\text{5th}} q \bigr) \qtq{and} g_\kappa(t) = g\bigl(x;\vk, e^{tJ\nabla H_\kappa} q \bigr).
\end{align}
\end{corollary}

\begin{proof}
Let us define
\begin{align*}
Q_* = \{ e^{tJ\nabla H_\kappa} q : t\in[-1,1] \text{ and } q\in Q\} \qtq{and} g_\text{diff}(t) = g\bigl(x;\vk, e^{tJ\nabla (H_\text{5th}-H_\kappa)} q \bigr).
\end{align*}
Then, by the commutativity of the flows, we need only show
\begin{align*}
\limsup_{\kappa\to\infty} \ \sup_{q\in Q_*}\  \sup_{|t|\leq 1} \  \relax
	\Bigl| \bigl\langle \phi,\bigl[ \tfrac{1}{g_\text{diff}(t)} - \tfrac{1}{g_\text{diff}(0)} \bigr] \bigr\rangle \Bigr|  =0.
\end{align*}
Also, from Proposition~\ref{P:H-1 bound}, we see that the set $Q_*$ inherits boundedness and equicontinuity from $Q$.  This will be important.

Beginning with \eqref{jj defn} and then applying Lemma~\ref{L:dLS 1} and Proposition~\ref{P:dLS}, we find
\begin{align*}
\bigl\| \bigl\langle \phi, \partial_t \tfrac{1}{g_\text{diff}(t)}\bigr\rangle \bigr\|_{L^1_t}
		\lesssim \kappa^{-\frac12} + \bigl\|\partial_x^2 \tfrac{1}{g_\text{diff}(t)}\bigr\|_{L^\infty_t H^{-1}_\kappa} ,
\end{align*}
uniformly for $q\in Q_*$ and $\kappa\geq\kappa_0$.  In this way, the proof of the proposition reduces to verifying the equicontinuity property
\begin{align}\label{1/g equi}
\limsup_{\kappa\to\infty} \ \sup_{q\in Q_*}\  \sup_{|t|\leq 1} \  \relax \bigl\|\partial_x^2 \bigl[\tfrac{1}{g_\text{diff}(t)}-2\vk\bigr]\bigr\|_{H^{-1}_\kappa} =0 .
\end{align}
The justification for calling this an equicontinuity property lies in the proof: Let
$$
Q_{**}=\bigl\{ e^{tJ\nabla (H_\text{5th}-H_\kappa)} q : t\in[-1,1],\ q\in Q_*\bigr\} \qtq{and} \mathcal G = \bigl\{  \tfrac{1}{g(\vk;q)}-2\vk : q\in Q_{**} \bigr\}.
$$
Then Proposition~\ref{P:H-1 bound} shows that $Q_{**}$ is $H^{-1}$-bounded and equicontinuous. It then follows that $\mathcal G$ is $H^1$-bounded and equicontinuous; this relies on both the diffeomorphism property (Lemma~\ref{L:diffeo}) and the fact that the mapping $q\mapsto (1/g-2\vk)$ commutes with translations.
This commutation property is not profound; it simply says that the Green's function for a translated potential is the corresponding translate of the original Green's function.  (On the other hand, even linear isomorphisms such as the Fourier transform on $L^2(\R)$ need not preserve equicontinuity.)

Returning now to \eqref{1/g equi}, we must show that
$$
\limsup_{\kappa\to\infty} \ \sup_{f\in\mathcal G}\ \int \frac{\xi^4 |\hat f(\xi)|^2}{\xi^2+\kappa^2}\,d\xi =0.
$$
This follows immediately from boundedness and equicontinuity.
\end{proof}

\section{Well-posedness}\label{S:WP}

In this section, we prove Theorem~\ref{T:main}.  This result will then allow us to upgrade the a priori bound proved in Section~\ref{S:LS} to a complete proof of Theorem~\ref{T:LS}.

\begin{proof}[Proof of Theorem~\ref{T:main}.]
Our immediate goal is to show that for any sequence of initial data $q_n\in B_\delta\cap\Schw(\R)$ that is $H^{-1}$-convergent, the corresponding solutions $q(t)$ to \eqref{5th} are Cauchy in $C_t H^{-1} ([-1,1]\times\R)$.  Here $\delta>0$ is assumed sufficiently small.  All claims in Theorem~\ref{T:main} can readily be deduce from this and the scaling transformation~\eqref{E:scaling}.

Mimicking the notations used in Corollary~\ref{C:dLS}, let us define
\begin{align}\label{gn flowed}
g_n(t) = g\bigl(x;\vk, e^{tJ\nabla H_\text{5th}} q_n \bigr) \qtq{and} g_{n,\kappa}(t) = g\bigl(x;\vk, e^{tJ\nabla H_\kappa} q_n \bigr),
\end{align}
where $\vk\geq 1$ is fixed here and for the remainder of the proof.

Given any $\phi\in\Schw(\R)$, we clearly have
\begin{align*}
\sup_{|t|\leq 1} \, \Bigl| \bigl\langle \phi,\bigl[ \tfrac{1}{g_n(t)} - \tfrac{1}{g_m(t)} \bigr] \bigr\rangle \Bigr| & \leq
	\sup_{|t|\leq 1} \, \Bigl| \bigl\langle \phi,\bigl[ \tfrac{1}{g_n(t)} - \tfrac{1}{g_{n,\kappa}(t)} \bigr] \bigr\rangle \Bigr| + 
	\Bigl| \bigl\langle \phi,\bigl[ \tfrac{1}{g_m(t)} - \tfrac{1}{g_{m,\kappa}(t)} \bigr] \bigr\rangle \Bigr| \\
&\qquad + \sup_{|t|\leq 1} \,  \Bigl| \bigl\langle \phi,\bigl[ \tfrac{1}{g_{n,\kappa}(t)} - \tfrac{1}{g_{m,\kappa}(t)} \bigr] \bigr\rangle \Bigr| .
\end{align*}
The significance of this is that by Corollary~\ref{C:dLS}, the first line can be made arbitrarily small (uniformly in $n$ and $m$) by choosing $\kappa$ sufficiently large.  Moreover, having chosen $\kappa$, the term on the second line can be made arbitrarily small by choosing $n$ and $m$ large enough; this is a consequence of the well-posedness of the $H_\kappa$ flow and the diffeomorphism property (Lemma~\ref{L:diffeo}).

Thus we have a form of weak-$H^1$ convergence of $\tfrac{1}{g_n} - 2\vk$ with some uniformity in $t$.  However, by Proposition~\ref{P:compact} and the diffeomorphism property,
$$
\mathcal G = \bigl\{  \tfrac{1}{g_n(t)}-2\vk : n\in \N \text{ and } t\in[-1,1] \bigr\}
$$
is precompact in $H^1$.  Thus (e.g., arguing by contradiction), we see that
$$
\tfrac{1}{g_n(t)}-2\vk \quad\text{is a Cauchy sequence in $C_t H^1([-1,1]\times\R)$.}
$$
Thus, by the diffeomorphism property, $q_n(t)$ is Cauchy in $C_t H^{-1} ([-1,1]\times\R)$.
\end{proof}

\begin{proof}[Proof of Theorem~\ref{T:LS}.]  In view of the scaling \eqref{E:scaling}, it suffices to prove all claims for small initial data.  Note that for such small data, the last term on RHS\eqref{E:T:LS} is redundant; it arises when undoing the scaling.

Given that solutions for general data are defined as limits of Schwartz solutions, it suffices to prove adequate estimates for such Schwartz solutions.  Concretely, we will show that for any sequence of initial data $q_n(0)\in B_\delta\cap\Schw(\R)$ that is $H^{-1}$-convergent, the corresponding solutions satisfy
\begin{align}\label{dinner}
\sup_{x_0\in\R} \ \iint \phi(x-x_0)^2 \bigl[ |(q_n-q_m)'(t,x)|^2 + |(q_n-q_m)(t,x)|^2 \bigr]\,dx\,dt \to0.
\end{align}
as $n,m\to\infty$.  As ever, the spacetime integral is over $[-1,1]\times\R$.
This is actually slightly stronger that what is needed to prove Theorem~\ref{T:LS}, due to the uniformity in $x_0$.  We have dropped the parameter $t_0$ here since this can be restored a posteriori by applying Theorem~\ref{T:main}.   We also recall that the boundedness of LHS\eqref{dinner} was shown already in Corollary~\ref{C:LS5}.
 
Adapting \eqref{breakfast} and \eqref{lunch} to our current setting and then applying Lemma~\ref{L:change}, we find that 
\begin{align*}
\text{LHS\eqref{dinner}} \lesssim  \vk^4  \| q_n-q_m \|_{L^\infty_t H^{-1}}^2
		+ \| q_n \|_{\vLS}^2  + \| q_m \|_{\vLS}^2
\end{align*}
uniformly in $\vk\geq 1$.  By Proposition~\ref{P:LS5} and the equicontinuity of $\{q_n(0)\}_{n\in\N}$, we see that the latter two terms can be made arbitrarily small (uniformly in $n$ and $m$) by choosing $\vk\geq 1$ sufficiently large.  But then Theorem~\ref{T:main} guarantees that the first summand can be made arbitrarily small by merely requiring $n$ and $m$ to be sufficiently large.  This proves \eqref{dinner}.

The fact that the solutions constructed in Theorem~\ref{T:main} are distributional solutions is readily deduced from the earlier parts of Theorem~\ref{T:LS}; see the discussion following the statement of Theorem~\ref{T:LS}.
\end{proof}

\end{document}